\documentclass[reqno,12pt]{amsart} 
\usepackage{graphicx}
\usepackage{amsmath,amsthm,amsfonts,amscd,amssymb,mathrsfs}
\usepackage{hyperref}
\hypersetup{colorlinks,linkcolor=black,citecolor=black, pdfstartview=FitH,pdfpagemode=FitScreen}

\usepackage{subfigure}%
\usepackage{etoolbox}
\let\bbordermatrix\bordermatrix
\patchcmd{\bbordermatrix}{8.75}{4.75}{}{}
\patchcmd{\bbordermatrix}{\left(}{\left[}{}{}
\patchcmd{\bbordermatrix}{\right)}{\right]}{}{}

\usepackage{color}
\usepackage{arydshln}

\numberwithin{equation}{section}
\newtheorem{theorem}{Theorem}[section]
\newtheorem{thm}[theorem]{Theorem}

\newtheorem{lemma}[theorem]{Lemma}
\newtheorem{lem}[theorem]{Lemma}

\newtheorem{cor}[theorem]{Corollary}
\newtheorem{ex}[theorem]{Example}
\newtheorem*{ex*}{Example}
\newtheorem{df}[theorem]{Definition}

\newcommand{\mtx}[1]{\left[\,\begin{matrix} #1\end{matrix}\,\right ]}       

\newcommand\diag{{\mbox{diag}\,}}

\newcommand\g{\mathfrak g}
\newcommand\s{\mathfrak s}

\newcommand\gl{\mathfrak {gl}}

\renewcommand\l{\mathfrak l}

\renewcommand\s{\mathfrak s}

\newcommand\R{\mathbb R}
\newcommand\C{\mathbb C}
\newcommand\N{\mathbb N}
\newcommand\Z{\mathbb Z}

\newcommand\F{\mathbb F}

\newcommand\ad{\operatorname{ad}}

\newcommand\tr{{\rm tr}}

\newcommand\End{\mbox{\rm End\,}}
\newcommand\Hom{\mbox{\rm Hom\,}}
\newcommand\Der{\mbox{\rm Der\,}}

\newcommand{\Ker}{\mbox{\rm Ker\,}}
\renewcommand{\Im}{\mbox{\rm Im\,}}

\newcommand{\mc}{\mathcal }

\newcommand{\nor}[1]{{\rm N}({#1})}
\newcommand{\cen}[1]{{\rm Z}({#1})}
\newcommand{\M}{{\mathcal M}}
\newcommand{\wt}[1]{\widetilde{#1}}

\newcommand{\E}[2]{E^{(#2)}_{#1}}	

\newcommand{\pr}[1]{\left( #1 \right)}
\newcommand{\vpr}[1]{\left\{ #1 \right\}}

\renewcommand{\L}{{\mathcal L}}

\begin{document}

\title[Derivations of Lie algebras of dominant upper triangular ladder matrices]{Derivations of   Lie algebras of dominant upper triangular ladder matrices}

\date{}

\author[P. Ghimire and H. Huang] {Prakash Ghimire, Huajun Huang}

\address{Department of Mathematics and Statistics, Auburn University, Auburn, AL 36849, USA}

\email{pzg0011@auburn.edu,  huanghu@auburn.edu}

\begin{abstract}
We explicitly describe the Lie algebras $M_{\L}$ of  ladder matrices in $M_n$ associate with dominant upper triangular ladders $\L$, and
completely characterize the derivations of  these  $M_{\L}$  over a field $\F$ with ${\rm char}(\F)\ne 2$. We also completely characterize the derivations of Lie algebras $[M_{\L}, M_{\L} ]$ where $\L$ are strongly dominant upper triangular ladders and ${\rm char}(\F)\ne 2, 3$. 
\end{abstract}


\maketitle


\section{Introduction}

Ladder matrix is a natural extension of block upper triangular matrix.   
A ladder matrix is one that has zero entries outside of a ladder shape region. 
Let $[n]:=\{1,2,\cdots,n\}.$  Given a field $\F$, let $M_{mn}$ be the set of   $m\times n$ matrices over   $\mathbb{F}$, and $M_n:=M_{nn}$.
We define a partial order on $\Z^+\times\Z^+$: 
$(i_1,j_1)$ is said to {\em dominate} $ (i_2,j_2)$, written as $(i_1,j_1)\succeq (i_2,j_2)$, whenever $i_1\ge i_2$ and $j_1\le j_2.$

\begin{df}
A subset $\L:=\{(i_1,j_1),\cdots, (i_s,j_s)\}$ of the index set $[n]\times [n]$ of $M_{n}$ is called  a {\em ladder} of {\em step} $s$ and {\em size} $n$, if
$$i_1<i_2<\cdots<i_s\qquad \text{and}\qquad j_1<j_2<\cdots<j_s.$$
Each $(i_\ell,j_\ell)$ ($\ell\in[s]$) is called a {\em corner point} of $\L$.
 The  set $M_{\L}$  of {\em $\L$-ladder matrices} is  a subset  of $M_n$  defined by $M_{\emptyset}=\{0\}$ and
$$
M_{\L} :=\sum_{(i,j)\in I(\L)} \F E_{ij},
$$
where 
$$
I(\L):=\{(i,j)\in [n]\times[n]: \text{$(i,j)\preceq (i_{\ell},  j_{\ell})$ for some $\ell\in [s]$}\},
$$
and $E_{ij}\in M_n$ denotes the $(i,j)$ standard matrix that has 1 as the $(i,j)$ entry and $0$ elsewhere.
\end{df}

In other words, $M_{\L}$ consists of matrices that have nonzero entries only  in the upper right direction of some coner points $(i_{\ell},j_{\ell})$ of $\L$.
In  \cite{BH}, Brice and Huang introduce the notion of ladder matrix and proved that
$M_{\L}\cdot M_{\L'}=M_{\L''},$
where $\L$ and $\L'$ are two arbitrary ladders of size $n$, and $\L''$ is a ladder decided by $\L$ and $\L'$.
In particular, if $\L$ is an upper triangular ladder (i.e., $i_{\ell}<j_{\ell+1}$ for $\ell\in[s-1]$, see Definition \ref{df-ladder}), then $M_{\L}$ is a matrix subalgebra of $M_{n}$.
Naturally, $M_{\L}$ is  a Lie subalgebra of $M_{n}$ (aka $\gl(n,\F)$) with respect to the standard Lie bracket $[X, Y]=XY-YX$.

Typical examples of    Lie algebras  $M_{\L}$ include those of  block upper triangular  matrices and of strictly block upper triangular matrices,
$M_{pq}$  embedded in the upper right corner of $M_n$ (when $p\le n$ and $q\le n$), and  $M_n$ itself. 
In 1957,  Dixmier and Lister constructed 
a nilpotent Lie algebra   \cite{DL} to disprove the converse of a statement of  Jacobson 
\cite{J}: ``a Lie algebra with a nonsingular derivation is nilpotent''; the corresponding derivation algebra is clearly embeded in a special nilpotent $M_{\L}$.

A derivation of Lie algebra $\g$ is a linear map $f\in\End(\g)$ that satisfies
$$f([X,Y])=[f(X),Y]+[X,f(Y)]\qquad\text{for all \ } X, Y\in\g.$$
The Lie derivations and generalized derivations of ladder shape matrix Lie algebras over a field or  ring has drawn much attension in recent years.  Here is a fairly imcompleted list of literatures.  
Chen determines the structure of certain generalized derivations of a parabolic subalgebra of
$\gl(n,\F)$ over a field $\F$  with ${\rm char} (\F)\ne 2$ and $|\F|>n\ge 3$ \cite{ZXC}.  Brice describes the derviations of parabolic 
subalgebra  of a reductive Lie algebra over an algebraically closed and characteristics zero field, and proves the zero-product determined property of such derivation algebras \cite{B}. 
Let $R$ be a communicative ring with identity. Cheung characterizes proper Lie derivations and gives sufficient conditions for any Lie derivation to be proper for triangular algebras over $R$ \cite{C}. 
Du and Wang investigate the Lie derivations of $2\times 2$ block generalized  matrix algebras \cite{DW}.
Wang, Ou, and Yu  describe the derivations of  intermediate Lie algebras 
between diagonal matrix algebra and upper triangular matrix algebra in $\gl(n, R)$ \cite {WOY}. 
Wang and Yu characterized all the derivations of parabolic subalgebras of $\gl(n,R)$ \cite{WQ}. 
Ou, Wang, and Yao describe the derivations of the Lie algebra of stictly upper triangular matrices in $\gl(n,R)$ \cite{OWY}.  
Ji, Yang, and Chen study the biderivations of the algebra of strictly upper triangular matrices in $\gl(n,R)$ \cite{JYC}. 
The Lie triple derivations are also extensively studied, for examples, on $\gl(n,R)$ \cite{KZL}, 
 on the algebra of upper triangular matrices of $\gl(n,R)$ \cite{Bd},
and on the parabolic subalgebras of $\gl(n,R)$ \cite{LCL}.

In this paper, we explicitly characterize the derivations of  the  Lie algebra $M_{\L}$  associate with a  dominant upper triangular (DUT) ladder $\L$
 for  ${\rm char}(\F)\ne 2$ (Theorem \ref{Main theorem}), and  the derivations of   $[M_{\L}, M_{\L}]$ associate with  a strongly dominant upper triangular (SDUT) ladder $\L$ for ${\rm char}(\F)\ne 2, 3$ (Theorem \ref{thm:strongly dominant}). 
 A ladder $\L=\{(i_1,j_1),\cdots,(i_s,j_s)\}$ is called DUT (resp. SDUT) if  $j_\ell\le i_\ell<j_{\ell+1}$ (resp. $j_\ell<i_\ell<j_{\ell+1}$) for $\ell\in[s-1]$. 
All $M_{\L}$ associate with  a DUT ladder $\L$ are completely characterized in Theorem \ref{thm:dominant-ladder}.  
\begin{itemize}
\item Theorem \ref{thm:dominant-ladder}: $\L$ is  DUT  if and only if $M_{\L}$
can be obtained by removing some non-consecutive diagonal blocks from the set of block upper triangular matrices corresponding to a partition of $[n]$.
\item
Theorem \ref{Main theorem}: When ${\rm char}(\F)\ne 2$ and $\L$ is a DUT ladder, 
every  derivation of  $M_{\L}$  is a sum of the adjoint action of a block upper triangular matrix and a linear map from $M_{\L}/[M_{\L},M_{\L}]$ to the center of $M_{\L}$. 
\item
Theorem \ref{thm:strongly dominant}: When ${\rm char}(\F)\ne 2,3$ and $\L$ is a SDUT ladder, 
every  derivation of  $[M_{\L},M_{\L}]$ is  the adjoint action of a block upper triangular matrix, so that it could be extended to a derivation of $M_{\L}$.
\end{itemize}
In general, a derivation of a Lie algebra stabilizes each subalgebra appearing in the derived series.  
Moreover,  the derived series of a non-solvable Lie algebra of upper triangular ladder matrices will terminate at $[M_{\L}, M_{\L}]$ 
for certain SDUT ladder $\L$.  
Therefore, knowledge on the derivation algebra of these $[M_{\L}, M_{\L}]$ would be useful
 to disclose the structure of derivations of Lie algebras of general upper triangular ladder matrices.

The paper is organized as follow: Section 2 provides some basic  properties of   ladder matrices; in particular, 
all DUT ladder matrix algebras are completely characterized  (Theorem \ref{thm:dominant-ladder}), and the   counting of these algebras in $M_n$ is done  (Corollary \ref{thm:counting-dominant}). Section 3 characterizes the derivations of $M_{\L}$ for DUT ladders $\L$ and  ${\rm char}(\F)\ne 2$   (Theorem \ref{Main theorem}), and gives examples and applications, e.g. on the derivations of step 1 ladder matrix algebras (Theorem \ref{thm:step 1}).  Section 4 gives the proof of the main theorem in Section 3. 
Section 5 determines  the derivations of $[M_{\L}, M_{\L}]$ for SDUT ladders $\L$ and ${\rm char}(\F)\ne 2, 3$.

\section{Preliminary}

We develop some basic properties of ladders and ladder matrices in this section. 
Given a ladder $\L\subset[n]\times[n]$, the matrices in $M_{\L}$  could be viewed as block matrices with respect to  suitable partitions.
 A {\em partition} of $[n]$ can be characterized by a subset
$$\gamma=\{i_1,i_2,\cdots,i_s\}\subseteq[n-1],\qquad i_1\le i_2\le\cdots\le i_s,$$
where the corresponding partition in $M_{n}$ is done {\em right after} the $i_1,i_2,\cdots,i_s$ rows
and columns.  Every ladder $\L$ corresponds to one simplest compatible partition  defined below.

\begin{df}\label{block matrix form}
Let $\L=\{(i_1,j_1),\cdots,(i_s,j_s)\}\subset [n]\times[n]$ be a ladder. 
{\em  The partition of $\L$}  is a partition of $[n]$ characterized by 
$$
  \gamma_{\L}:=\{i_1,i_2,\cdots,i_s\}\cup\{j_1-1,j_2-1,\cdots,j_s-1\}-\{0,n\}.
  $$
The  matrices in $M_{\L}$ could be viewed as block matrices with respect to the partition $\gamma_{\L}$. Denote by $[I(\L)]$
{\em the index set of nonzero blocks of  $M_{\L}$ with respect to  $\gamma_{\L}$}. 
If we set $t:=|\gamma_{\L}|+1$, then $[I(\L)]\subset [t]\times [t]$.
\end{df}

The set of block upper triangular matrices corresponding to  a partition $\gamma_{\L}=\{n_1,\cdots,n_{t-1}\}$ is exactly  $M_{\L_{B}}$,  where
\begin{equation}
\L_{B}=\{(n_1,1), (n_2,n_1+1),\cdots,(n_{t-1},n_{t-2}+1), (n, n_{t-1}+1)\}.
\end{equation}

We introduce some special ladders to be used in the paper.

\begin{df}\label{df-ladder}
A ladder $\L=\{(i_1,j_1),\cdots,(i_s,j_s)\}$  in $[n]\times [n]$ is called
\begin{itemize}
\item
{\em upper triangular:} if $i_{\ell}<j_{\ell+1}$ for  $\ell\in[s-1]$;
\item
{\em strictly  upper triangular:} if $i_{\ell}<j_{\ell}$ for $\ell\in[s]$;
\item
{\em dominant upper triangular (DUT):} if $j_\ell\le i_\ell<j_{\ell+1}$ for $\ell\in[s-1]$;
\item
{\em strongly dominant upper triangular (SDUT):} if $j_\ell<i_\ell<j_{\ell+1}$ for $\ell\in[s-1]$.
\end{itemize}
When $\L$ is upper triangular, a matrix in $M_{\L}$ is called an {\em upper triangular ($\L$-) ladder matrix}.
Similarly for the others.
\end{df}

The above different kinds of   ladder $\L$ can be easily distinguished by the shape of $M_{\L}$. 
They can also be
reinterpreted by the block form of $M_{\L}$ with respect to the partition $\gamma_{\L}$:
\begin{itemize}
\item
$\L$ is  upper triangular if $M_{\L}\subseteq M_{\L_{B}}$ (resp. $I(\L)\subseteq I(\L_{B})$);
\item
$\L$ is strictly upper triangular if  $M_{\L}$ is contained in the strictly block upper triangular part of $M_{\L_{B}}$;
\item
$\L$ is DUT if every block index $(i,j)\in [I(\L)]$ is dominated by a diagonal one $(k,k)\in [I(\L)]$;
\item
$\L$ is SDUT if $\L$ is DUT, and every nonzero diagonal block in $M_{\L}$ has size greater than 1.
\end{itemize}


\begin{ex} \label{ex:ladder-size7-11-43-55}
Consider the ladder $\L=\{(1,1),(4,3),(5,5)\}$ of size $7$. Then $\L$ is DUT but not SDUT.  The matrix form of $M_\L$ is given in Figure \ref{fig:ladder-size7-11-43-55-form}.
The index set $I(\L)$ of $\L$ consists of $(i,j)\in[7]\times[7]$ dominated by at least one of $(1,1)$, $(4,3)$, and $(5,5)$:
$$
I(\L)=\{(1,j): 1\le j\le 7\}\cup\{(i,j): 2\le i\le 4,\ 3\le j\le 7\}\cup\{(5,j): 5\le j\le 7\}.
$$
The partition of $\L$ is given by
$$\gamma_{\L}=\{1,4,5\}\cup\{1-1,3-1,5-1\}-\{0,7\}=\{1,2,4,5\}.$$
So matrices in $M_{\L}$ are partitioned after the 1, 2, 4, 5 rows and columns.  Figure \ref{fig:ladder-size7-11-43-55-block-form} indicates the block form of $M_{\L}$.
The block index set $[I(\L)]$ consists of $(i,j)\in[5]\times[5]$ dominated by at least one of $(1,1)$, $(3,3)$, and $(4,4)$:
$$
[I(\L)]=\{(1,j): 1\le j\le 5\}\cup\{(i,j): 2\le i\le 3,\ 3\le j\le 5\}\cup\{(4,4), (4,5)\}.
$$

\begin{figure}[ht]
\centering
\subfigure[matrix form of  $M_{\L}$] {
$\begin{array}{|c c c cc cc|}
\hline
* &* &* &* &* &*&*
\\ \cline {1-2}
0 &\multicolumn{1}{r|}{0} &* &* &* &*&*
\\
0 & \multicolumn{1}{r|}{0}&* &* &* & *&*
\\
0 & \multicolumn{1}{r|}{0} &*  &* &* & *&*
\\  \cline {3-4}
0 &0 &0 &\multicolumn{1}{r|}{0} &* & *&*
\\ \cline{5-7}
0 &0 &0 &0 &0 & 0&0
\\
0 &0 &0 &0 & 0 & 0&0
\\ \hline
\end{array}
$
\label{fig:ladder-size7-11-43-55-form}
}
\qquad\qquad
\subfigure[block matrix form of  $M_{\L}$]{%
$\begin{array}{|c: c:c c:c: c c|}
\hline
*& * &* &* &* &* &*
\\ \hdashline
0 &0 &* &* &* &* &*
\\ \hdashline
0 &0 &* &* &* &* &*
\\
0 &0 &* &* &* &* &*
\\ \hdashline
0 &0 &0 &0 &* &* &*
\\ \hdashline
0 &0 &0 &0 &0 &0 &0
\\
0 &0 &0 &0 &0 &0 &0
\\ \hline
\end{array}
$
\label{fig:ladder-size7-11-43-55-block-form}
}
\caption{Ladder $\L=\{(1,1), (4,3), (5,5)\}$ of size $7$}
\end{figure}
\end{ex}

Now we can completely characterize DUT ladders and ladder matrices in terms of the associated partition.

\begin{thm} \label{thm:dominant-ladder}
\begin{enumerate}
\item
A ladder $\L$ is DUT if and only if  for each $i\in [t-1]$, at least one of $(i,i)$ and $(i+1,i+1)$ is in $[I({\L})]$. 
In particular, $\L$ is DUT implies that $(i,j)\in [I({\L})]$ for any $i,j\in[t]$ with $i<j$. 
\item
Equivalently, $M_{\L}$ is a set of DUT ladder matrices, if and only if it
can be obtained by removing some non-consecutive diagonal blocks from the set of block upper triangular matrices corresponding to a partition of $[n]$.
\end{enumerate}
\end{thm}

\begin{proof} It suffices to prove the first statement. 
Let $\L=\{(i_1,j_1),\cdots,(i_s,j_s)\}$ and $\gamma_{\L}=\{n_1,\cdots,n_{t-1}\}$, so that the set of block upper triangular matrices is $M_{\L_{B}}$ for the ladder
$\L_{B}=\{(n_1,1), (n_2,n_1+1),\cdots,(n_{t-1},n_{t-2}+1), (n, n_{t-1}+1)\}$. 

Suppose for each $i\in [t-1]$, at least one of $(i,i)$ and $(i+1,i+1)$ is in $[I({\L})]$. 
Then $\L\subseteq\L_{B}$.  It is obvious that $j_\ell\le i_\ell<j_{\ell+1}$ for $\ell\in[s-1]$. 
Hence $\L$ is DUT. 

Now assume that $\L$ is DUT.  Then every block index $(i,j)\in[I(\L)]$ is dominated by a diagonal block index $(k,k)\in [I(\L)]$, that is, $i\le k\le j$. 
If for some $i\in[t-1]$, neither $(i,i)$ nor $(i+1,i+1)$ is in $[I(\L)]$, then $(i,i+1)$ is not in $[I(\L)]$. Then $n_i\not\in\gamma_{\L}$, 
which is a contradiction.  Therefore, at least one of $(i,i)$ and $(i+1,i+1)$ is in $[I({\L})]$. 
\end{proof}

A direct application of Theorem \ref{thm:dominant-ladder} is the counting of DUT   ladder matrices.

\begin{cor}\label{thm:counting-dominant}
Let $\{F_t\}_{t=1}^\infty=\{1,1,2,3,5,\cdots\}$ be the Fibonacci sequence.
\begin{enumerate}
\item
The number   of sets of DUT ladder matrices corresponding to a  $t\times t$ block form equals to $b_{t}$, where
\begin{equation}\label{b_t}
\{b_t\}_{t=1}^\infty=\{F_{t+2}\}_{t=1}^\infty=\{2, 3, 5, 8,13,\cdots\}.
\end{equation}

\item
The number of  sets of DUT ladder matrices  in $M_{n}$ equals to $a_n$, where
\begin{equation}
\{a_n\}_{n=1}^\infty=\{F_{2n+1}\}_{n=1}^\infty=\{2, 5, 13, 34, 89, \cdots\}.
\end{equation}

\end{enumerate}
\end{cor}

\begin{proof}
\begin{enumerate}

\item
Clearly
$b_1=2=F_3$ and $b_2=3=F_4$.
  \eqref{b_t} will be proved if  $\{b_t\}$ satisfies the same recursive formula as $\{F_{t+2}\}$ does,  that is,
\begin{equation}\label{b_t-rel}
b_t=b_{t-1}+b_{t-2}.
\end{equation}
 By Theorem \ref{thm:dominant-ladder},   $b_t$  equals to the number of ways to choose non-consecutive diagonal blocks in a given $t\times t$ block  form.  If the first diagonal block is chosen, then the second one should be skipped, and there are $b_{t-2}$ ways to choose the remaining diagonal blocks; if the first diagonal block is not chosen, then there are  $b_{t-1}$ ways  to choose the remaining diagonal blocks. Therefore, \eqref{b_t-rel} is true and \eqref{b_t} is proved.

\item
Given $t\in [n]$, there are $\binom{n-1}{t-1}$ ways to partition matrices in $M_n$ into a $t\times t$ block   form; each block form corresponds to
$b_t=F_{t+2}$  sets of DUT ladder matrices.
Let  $r_1:=\frac{1+\sqrt{5}}{2}$ and $r_2:=\frac{1-\sqrt{5}}{2}$ be the roots of $x^2-x-1=0$.
The Binet's Fibonacci number formula says that
$$
F_t=\frac{1}{\sqrt{5}}r_1^t -\frac{1}{\sqrt{5}}r_2^t.
$$
Therefore,
\begin{eqnarray*}
a_n &=& \sum_{t=1}^{n} \binom{n-1}{t-1} F_{t+2}
=
\sum_{t=1}^{n} \binom{n-1}{t-1} \left(\frac{1}{\sqrt{5}}r_1^{t+2} -\frac{1}{\sqrt{5}}r_2^{t+2} \right)
\\
&=&
\frac{1}{\sqrt{5}}\left[r_1^3(1+r_1)^{n-1}-r_2^3(1+r_2)^{n-1}\right]
=\frac{1}{\sqrt{5}}\left[r_1^3(r_1^2)^{n-1}-r_2^3(r_2^2)^{n-1}\right]
=F_{2n+1}. \qedhere
\end{eqnarray*}
\end{enumerate}
\end{proof}

We give some notations that will be used in  studying the partitioned matrices associated with $M_{\L}$. 

\begin{df}
  Given an algebra $(M,+,*)$ and two subsets $M', M''\subseteq M$,
  define the subset
  $$M'*M'':=\left\{\sum_{i=1}^m A_i*B_i\mid m\in\N, A_i\in M', B_i\in M''\right\}.$$
\end{df}

\begin{df}\label{df:matrix}
Consider the matrices in  $M_n$ with respect to a given partition $\gamma_{\L}$. 
\begin{itemize}
\item
Let $\M_{ij}$ denote the set of all submatrices in the $(i,j)$ block of $M_n$.  
Let $E_{pq}^{[ij]}$  denote  the $(p,q)$ standard matrix in  $\M_{ij}$.
\item
Let  $\widetilde{\M_{ij}}$ denote the embedding of $\M_{ij}$ in $M_n$. 
\item
For $A\in M_n$, let $A_{ij}\in {\M}_{ij}$ denote the $(i,j)$ block submatrix of $A$.
\item
For $B_{ij}\in {\M}_{ij}$, let $\widetilde{B_{ij}}\in \widetilde{\M_{ij}}$ denote the matrix in $M_n$ with $B_{ij}$ in the $(i,j)$ block and zero elsewhere. 
Similarly for $\wt{B_{ik}B_{kj}}$ if $B_{ik}\in \M_{ik}$ and $B_{kj}\in \M_{kj}$ . 
\item
In $\M_{kk}$, let $I_{kk}$ denote the identity matrix, and $\s\l_{kk}$ the set of traceless matrices, respectively.
\end{itemize}
A notation of double index, say $\M_{ij}$, may be written as $\M_{i,j}$  for clarity purpose. 
\end{df}

The {\em normalizer} $\nor{M_{\L}}$ and the {\em centralizer} $\cen{M_{\L}}$ of Lie subalgebra $M_\L$ in $M_n$ are:
\begin{eqnarray*}
\nor{M_{\L}} &=& \{A\in M_n: [A, B]\in M_{\L} \text{ for all } B\in M_{\L}\},
\\
\cen{M_{\L}} &=& \{A\in M_n: [A,B]=0  \text{ for all } B\in M_{\L}\}.
\end{eqnarray*}
They are explcitly described by the following two lemmas. 

\begin{lem}\label{thm:normalizer}
If $\L$ is a DUT ladder, then $\nor{M_{\L}}=M_{\L_B}$, 
the subalgebra of block upper triangular matrices with respect to the partition of $\L$.
\end{lem}

\begin{proof}
We first show that $\nor{M_{\L}}\subseteq M_{\L_B}$. Suppose on the contrary, there is $A\in \nor{M_{\L}}$ such that the $(i,j)$ block $A_{ij}\ne 0$ for some $i>j$.
There are two cases:
\begin{enumerate}
\item
$i>j+1$: We have $\widetilde{\M_{j,j+1}}\subseteq M_{\L}$   by Theorem \ref{thm:dominant-ladder}.  So
$[A, \widetilde{\M_{j,j+1}}]\subseteq M_{\L}$.  However, its $(i,j+1)$ block is
$$
[A, \widetilde{\M_{j,j+1}}]_{i,j+1}=[A_{ij}, \M_{j,j+1}]=A_{ij}\M_{j,j+1}\ne \{0\},
$$
which  contradicts to the DUT assumption of $\L$.
\item
$i=j+1$: By Theorem \ref{thm:dominant-ladder}, either $\wt{\M_{jj}}\subseteq M_{\L}$ or $\wt{\M_{j+1,j+1}}\subseteq M_{\L}$.
Without loss of generality, suppose $\wt{\M_{jj}}\subseteq M_{\L}$.  Then
$[A, \wt{\M_{jj}}]\subseteq M_{\L}$. However, its $(i,j)$ block is
$$
[A, \wt{\M_{jj}}]_{ij}=A_{ij}\M_{jj}\ne\{0\},
$$
which  contradicts the DUT assumption of $\L$.
\end{enumerate}
Therefore,  $A\in M_{\L_B}$ and thus $\nor{M_{\L}}\subseteq M_{\L_B}$.

 For any $(i,j)\in [t]\times[t]$ with $i\le j$, 
the possibly nonzero blocks of matrices in $[\wt{\M_{ij}},M_{\L}]$ are those $(i,q)$ blocks with $q\ge j$ and $(p,j)$ blocks with $p\le i$, all of which belong to $M_{\L}$.  Hence $M_{\L_B}\subseteq \nor{M_{\L}}$. 
\end{proof}

\begin{lem} \label{thm:centralizer}
Let $\L$ be a DUT ladder and $t=|\gamma_{\L}|+1$.  
\begin{enumerate}
\item
If both the $(1,1)$ and the $(t,t)$ blocks of $M_{\L}$ are zero, then $\cen{M_{\L}}=\F I_n+\wt{\M_{1t}}$.
\item
Otherwise, $\cen{M_{\L}}=\F I_n.$
\end{enumerate}
\end{lem}

\begin{proof}
Clearly $\cen{M_{\L}}\subseteq \nor{M_{\L}}$.  The possibly nonzero blocks of any $A\in\cen{M_{\L}}$ are $A_{ij}$ for some $1\le i\le j\le t$. 
If $A_{ij}\ne 0$ and $2\le i<j$, then $\wt{\M_{i-1,i}}\subseteq M_{\L}$, and we can find $B_{i-1,i}\in \M_{i-1,i}$ such that
$$
0\ne B_{i-1,i}A_{i,j}=[B_{i-1,i},A_{i,j}]=[\wt{B_{i-1,i}},A]_{i-1,j},
$$
which contradicts to the assumption $A\in\cen{M_{\L}}$. Thus $A_{ij}=0$ for all $2\le i<j\le t$. Similarly, $A_{ij}= 0$ for all $1\le i<j\le t-1$. 
So the only possibly nonzero blocks of $A\in\cen{M_{\L}}$  are $A_{1t}$ and $A_{ii}$ for $i\in [t]$. 

If $(1,1)\in[I(\L)]$, then $0=[\wt{I_{11}},A]_{1t}=A_{1t}.$  Similarly, $(t,t)\in [I(\L)]$ implies that $A_{1t}=0$. 
If neither $(1,1)$ nor $(t,t)$ is in $[I(\L)]$, then  $\wt{\M_{1t}}$ is in $\cen{M_{\L}}$ by direct computation. 

Now for any $i,j\in[t]$ with $i<j$ and  $\wt{B_{ij}}\in\wt{\M_{ij}}\subseteq M_{\L}$,  
$$
0=[A, \wt{B_{ij}}]_{ij}=A_{ii}B_{ij}-B_{ij}A_{jj}.
$$
Let $B_{ij}$ go through all standard matrices in $\M_{ij}$ that have  an entry one and zeros elsewhere. 
We can get $A_{ii}=\lambda I_{ii}$ and $A_{jj}=\lambda I_{jj}$ for a fixed $\lambda\in\F$.

In summary,    $\cen{M_{\L}}$ is described by the statements (1) and (2).
\end{proof}

\section{The main theorem}

In this section, we explicitly characterize the derivation algebra $\Der(M_{\L})$  for any DUT ladder $\L$  over a field $\F$ with ${\rm char}(\F)\ne 2$, and provide some consequent results.
Note that the  adjoint representation   $\ad: M_n\to \Der(M_n)$ defined by $\ad A(B)=[A,B]$ induces a Lie algebra homomorphism
$$\ad(\cdot)|_{M_{\L}}: \nor{M_{\L}}/\cen{M_{\L}}\to \Der(M_{\L}),$$
which will be used in the following theorem. 

\begin{theorem}{(Main theorem)} \label{Main theorem}
Suppose ${\rm char}(\F)\ne 2$.  Let $\L$ be a DUT ladder. Then the Lie algebra $\Der(M_{\L})$ can be decomposed as  a direct sum of ideals:
\begin{eqnarray} 
\label{Der-M_L}
\Der(M_{\L})  &=& \ad(\nor{M_{\L}}/\cen{M_{\L}})|_{M_{\L}}\oplus\mathcal{D}
\\ \label{Der-M_L-2}
&=& \left.\left(\ad \left(\frac{M_{\L}}{\cen{M_{\L}}\cap M_{\L}}\right)\rtimes \bigoplus_{(k,k)\in [I(\L_{B})]-[I(\L)]} \ad \left(\wt{\M_{kk}}\right)\right)\right|_{M_{\L}}\oplus\mathcal{D}
\end{eqnarray}
where
\begin{itemize}
\item
the normalizer $\nor{M_{\L}}$ and the centralizer $\cen{M_{\L}}$ are described by  Lemmas \ref{thm:normalizer} and \ref{thm:centralizer}, respectively;
\item the ideal ${\mc D}$ is defined by
\begin{equation}\label{Der-D}
\mathcal{D}:=\{\phi\in \End(M_{\L}): \Ker\phi\supseteq [M_{\L},  M_{\L}],\ \Im\phi\subseteq \cen{M_{\L}}\cap M_{\L}\};
\end{equation}
in particular, $\mathcal{D}\simeq \Hom_{\F}\left(M_{\L}/[M_{\L}, M_{\L}],  \cen{M_{\L}}\cap M_{\L}\right)$ as vector spaces.
\end{itemize}
Explicitly,  we have the following cases  with respect to the partition $\gamma_{\L}$ of $\L$ (let $t=|\gamma_{\L}|+1$):
\begin{enumerate}
\item If  $M_\L$ is a set of block upper triangular matrices (i.e. ${\L}={\L_B}$), then  every $f\in \Der(M_{\L})$ corresponds to an $X\in M_{\L_B}/\F I_{n}$  and $c_1,\cdots,c_t\in\F$,  such that
\begin{equation}\label{Der-block-upper-triagular}
   f(A)=\ad{X}(A)+\left(\sum_{k\in [t]} c_k \tr(A_{kk})\right)I_n\qquad \text{for }\quad A\in M_{\L}.
\end{equation}
\item If $M_{\L}$ has some zero diagonal block(s), but at least one of its $(1,1)$ and $(t,t)$ blocks is nonzero,
then  every $f\in \Der(M_{\L})$ corresponds to an $X\in M_{\L_{B}}/\F I_n$, such that
\begin{equation}\label{Der-with-first-or-last}
f(A)=\ad{X}(A)\qquad \text{for}\quad  A\in M_{\L}.
\end{equation}
\item
If both the $(1,1)$ and the $(t,t)$ blocks  in $M_{\L}$ are zero, then every $f\in \Der(M_{\L})$ corresponds to an $X\in M_{\L_B}/(\F I_n+\wt{\M_{1t}})$ and
$Y_{1tk}\in \M_{1t}$ for each $(k,k)\in [I(\L)]$, such that
\begin{equation}\label{Der-no-first-and-last}
    f(A)=\ad{X}(A)+ \sum_{(k,k)\in [I(\L)]} \tr(A_{kk})\wt{Y_{1tk}}\qquad \text{for}\quad  A\in M_{\L}.
\end{equation}
\end{enumerate}
\end{theorem}

A detailed proof of Theorem \ref{Main theorem} will be given in  Section \ref{sect:proof}.
The special case  $\L=\L_B$ (where $M_{\L}$ is a set of block upper triangular matrices) is included in a paper of Dengyin Wang and Qiu Yu \cite[Theorem 4.1]{WQ}.
Moreover,  Daniel Brice  obtains a  formula similar to \eqref{Der-M_L} for the  derivation algebra  of the parabolic subalgebra of a reductive Lie algebra over a $\C$-like fields or over $\R$  \cite{B}.

\begin{ex}\label{ex-main-char-2}
Theorem  \ref{Main theorem} is  not true when ${\rm char}(\F)=2$. Consider $M_{\L}=M_2$ with the basis ${\mathcal B}=\{E_{11}, E_{12}, E_{21}, E_{22}\}$. 
Define
$f\in  \End(M_{\L})$ by $f(E_{12})=E_{21}$ and
$f(E_{ij})=0$ for $(i,j)=(1,1), (2,1), (2,2)$. 
It is straightforward to verify that 
\begin{equation}\label{char-2-der}
f([E,E'])=[f(E),E']+[E,f(E')]
\end{equation}
for any $E,E'\in{\mathcal B}$, since there are only two cases that either side of \eqref{char-2-der} is nonzero:
$\{E, E'\}=\{E_{11}, E_{12}\}$ or $\{E_{12}, E_{22}\}.$
Therefore $f\in \Der(M_{\L})$.  
However,  $f$ is not an element of $\ad(\nor{M_{\L}}/\cen{M_{\L}})|_{M_{\L}}\oplus\mathcal{D}$ in \eqref{Der-M_L}.
\end{ex}

When $\L$ is an upper triangular ladder, an inner derivation $(\ad X)|_{M_{\L}}$  ($X\in M_{\L}$) satisfies that  
$$(\ad X)|_{M_{\L}}(\wt{\M_{ij}})\subseteq \wt{\M_{ij}}+\sum_{k> j} \wt{\M_{ik}}+\sum_{\ell<i} \wt{\M_{\ell j}}
\quad\text{for any}\quad \wt{\M_{ij}}\subseteq M_{\L}.
$$
So the inner derivation sends the $(i,j)$ block to a sum of blocks with the indices dominated by $(i,j)$. 
This dominance property also holds for {\em all} derivations of $M_{\L}$ when  $\L$ is  a DUT ladder with some zero diagonal blocks (i.e. $M_{\L}\ne M_{\L_B}$).

\begin{cor}\label{Cor}
Let $\L$ be a DUT ladder with some zero diagonal blocks. Then every $f\in Der(M_{\L})$ maps any
$(i,j)$ block of $M_{\L}$ to a sum of some blocks dominated by the $(i,j)$ block.
\end{cor}

\begin{proof}
The corollary is a direct consequence of Theorem \ref{Main theorem}(2) and (3).
\end{proof}

In general,  Corollary \ref{Cor} may not be true if $\L$ is not a DUT ladder which can be seen via the following example.

\begin{ex}
  Suppose $\F$ is an arbitrary field. Let $n=5$ and $\L=\{(1,2), (3,4)\}$. Then $\L$ is not DUT, and $M_{\L}$ has the form:
$$\left(
\begin{array}{c|c c|c c}
0&a_{12} &a_{13}&a_{14} &a_{15}\\
\hline
0 & 0& 0& a_{24}&a_{25}\\
0 & 0& 0& a_{34}& a_{35}\\
\hline
0 & 0& 0& 0& 0\\
0 & 0 &0& 0 &0\\
\end{array}
\right), \quad a_{ij}\in \mathbb{F}.$$
So $M_{\L}$ has a basis ${\mathcal B}=\{E_{12}, E_{13}, E_{14},E_{15},E_{24}, E_{25}, E_{34}, E_{35}\}$. 
Given $a,b\in \mathbb{F}$, define $f\in \End (M_{\L})$ by
\begin{align*}
f(E_{12}):&=\left(
\begin{array}{c|c c|c c}
0&0 &0&0& 0\\
\hline
0 & 0& 0& 0& 0\\
0 & 0 &0& a& b\\
\hline
0 & 0& 0& 0& 0\\
0 & 0& 0& 0& 0\\
\end{array}
\right),\ \
f(E_{13}):=\left(
\begin{array}{c|c c|c c}
0&0 &0&0& 0\\
\hline
0 & 0 &0& a& b\\
0 & 0& 0& 0 &0\\
\hline
0 & 0 &0& 0& 0\\
0 & 0& 0& 0 &0\\
\end{array}
\right),
\end{align*}
and $f(E)=0$  for all other matrices $E$ in the basis ${\mathcal B}$. We prove that
\begin{equation}\label{counter example 1}
f([E,E'])=[f(E),E']+[E,f(E')]\quad\text{for all}\quad E,E'\in{\mathcal B},
\end{equation}
so that $f$ is a derivation of $M_{\L}$. 
On one hand, $[E, E^{'}]\in {\rm span}\{E_{14}, E_{15}\}$ and thus $f([E, E^{'}])=0$;
on the other hand, in \eqref{counter example 1},
$[f(E), E^{'}]\ne 0$ or $[E, f(E^{'})]\ne 0$   only when
$\{E, E^{'}\}=\{E_{12}, E_{13}\}$, for which
 the equality \eqref{counter example 1} is easily verified. 
 Therefore, $f\in\Der(M_{\L})$. 
 However, $f$ maps the block $\wt{\M_{12}}$ into $\wt{\M_{23}}$, where $(2,3)$ is not dominated by $(1,2)$. 
\end{ex}

An important family of ladders is that of 1-step ladders $\L=\{(i,j)\}$, where   each $M_{\L}$ realizes
$M_{pq}$ ($p,q\le n$) as a Lie subalgebra of $M_n$.  
Many 1-step ladders are DUT.  
The derivations of these $M_{\L}$ can be explicitly characterized here.

\begin{theorem}\label{thm:step 1}
Let $\L=\{(i,j)\}\subseteq [n]\times[n]$ be a $1$-step ladder of size $n$.
\begin{enumerate}
\item If $i<j $, then $M_{\L}$ is abelian and
\begin{align*}
  \Der(M_{\L})&=\End(M_{\L}).
  \end{align*}
\item If $i=n$ or $j=1$, then
\begin{align*}
\Der(M_{\L})=\ad(\nor{M_{\L}}/\cen{M_{\L}})|_{M_{\L}}.
\end{align*}
Explicitly, there are three subcases:
\begin{enumerate}
 \item If $i=n$ and $j=1$, then $M_{\L}=M_n$, and $\Der(M_{\L})=\ad(M_n/\F I_n).$
 \item If $i\ne n$ and $j=1$, then
 \begin{eqnarray*}
 M_{\L} &=& \left\{\mtx{A_{11} &A_{12}\\0 &0}\in M_n: A_{11}\in M_{i},\ A_{12}\in M_{i,n-i}\right\},
 \\
 \Der(M_{\L})  &=& \left.\ad \left\{\mtx{X_{11} &X_{12}\\0 &X_{22}}: X_{11}\in M_{i},\ X_{12}\in M_{i,n-i},\ X_{22}\in M_{n-i}\right\}\right|_{M_{\L}}
 \\
 &=& \ad M_{\L}\rtimes \ad (\wt{\M_{22}})|_{M_{\L}}.
 \end{eqnarray*}
 \item If $i=n$ and $j\ne 1$, then
 \begin{eqnarray*}
 M_{\L} &=& \left\{\mtx{0 &A_{12}\\0 &A_{22}}\in M_n: A_{12}\in M_{j-1,n-j+1},\ A_{22}\in M_{n-j+1}\right\},
 \\
 \Der(M_{\L}) &=& \left.\ad \left\{\mtx{X_{11} &X_{12}\\0 &X_{22}}: X_{11}\in M_{j-1},\ X_{12}\in M_{j-1,n-j+1},\ X_{22}\in M_{n-j+1}\right\}\right|_{M_{\L}}
 \\
 &=& \ad M_{\L}\rtimes \ad (\wt{\M_{11}})|_{M_{\L}}.
 \end{eqnarray*}

\end{enumerate}
\item If $n>i\ge j>1$. Then
\begin{align*}
\Der(M_{\L})=\ad\left(\nor{M_{\L}}/\cen{M_{\L}}\right)|_{M_{\L}}\oplus\mathcal{D}
\end{align*}
where ${\mc D}$ is defined in \eqref{Der-D}.  Explicitly, $\gamma_{\L}=\{j-1,i\}$,   and
\begin{eqnarray*}
 M_{\L} &=& \left\{A=\mtx{0 &A_{12} &A_{13}\\0 &A_{22} &A_{23}\\0 &0 &0}\right\},
 \\
 \Der(M_{\L}) &=& \left.\ad \left\{\mtx{X_{11} &X_{12} &X_{13}\\0 &X_{22} &X_{23}\\0 &0 &X_{33}}\right\}\right|_{M_{\L}}
		 \oplus  \left\{f_Y: f_Y(A)=\tr(A_{22})\mtx{0 &0 &Y\\0 &0 &0\\0 &0 &0}\right\}
 \\
 &=& \left( \ad M_{\L} \rtimes \left.\left(\ad (\wt{\M_{11}})\oplus \ad (\wt{\M_{33}})\right)\right|_{M_{\L}} \right)
 \oplus  \left\{f_Y: f_Y(A)=\tr(A_{22})\mtx{0 &0 &Y\\0 &0 &0\\0 &0 &0}\right\}.
\end{eqnarray*}
\end{enumerate}
\end{theorem}

\begin{proof}
The cases (2) and (3) are done by Theorem \ref{Main theorem}. For case (1) where $M_{\L}$ is abelian, every $f\in\End(M_{\L})$ satisfies that
$$f([A,B])=0=[f(A),B]+[A,f(B)],\qquad A, B\in M_{\L}.$$
Therefore, $\Der(M_{\L})=\End(M_{\L})$.
\end{proof}


\section{Proof of the main theorem}\label{sect:proof}

To prove Theorem \ref{Main theorem}, we give several auxiliary results here.
The first two lemmas below connect the linear transformations within the four blocks of a $2\times 2$ block matrix:
$$
\bbordermatrix{~ & p & q \cr
              m & M_{mp} & M_{mq}  \cr
              n & M_{np} & M_{nq} \cr}
$$
Let $\E{ij}{mn}$ denote the $(i,j)$ standard matrix in $M_{mn}$.

\begin{lemma}\label{thm:block-tran-1}
 Suppose $\F$ is an arbitrary field. If linear transformations $\phi:M_{mp}\to M_{mq}$ and $\varphi: M_{np}\to M_{nq}$ satisfy that
\begin{equation}\label{block-tran-eq1}
\phi(AB)=A\varphi(B)\qquad\text{for all}\quad A\in M_{m n},\ \ B\in M_{n p},
\end{equation}
then  there is $X\in M_{pq}$ such that
$\phi(C)=CX$ for $C\in M_{mp}$ and $\varphi(D)=DX$ for $D\in M_{np}$.
\end{lemma}

\begin{proof}
For any  $j\in [n]$ and any $B\in M_{np}$,
$$\phi(\E{1j}{mn} B)=\E{1j}{mn}\varphi(B).$$
All such $\E{1j}{mn}B$ span the first row space of $M_{mp}$.   So $\phi$ maps the first row of $M_{mp}$ to the first row of $M_{mq}$.  There exists a unique $X\in M_{pq}$ such that
$$\E{1j}{mn}\varphi(B)=\phi(\E{1j}{mn}B)=\E{1j}{mn}BX,\qquad\text{for all \ } j\in[n],\ B\in M_{np}.$$
Therefore, $\varphi(B)=BX$.  Then $\phi(AB)=A\varphi(B)=ABX$ for any $A\in M_{mn}$ and $B\in M_{np}$.  Hence
$\phi(C)=CX$ for all $C\in M_{mp}$.
\end{proof}

\begin{lemma}\label{thm:block-tran-2}
 Suppose $\F$ is an arbitrary field. If linear transformations $\phi:M_{mp}\to M_{np}$ and $\varphi: M_{mq}\to M_{nq}$ satisfy that
\begin{equation}\label{block-tran-eq2}
\phi(BA)=\varphi(B)A\qquad\text{for all}\quad  A\in M_{qp},\ B\in M_{mq},
\end{equation}
then  there is $X\in M_{nm}$ such that
$\phi(C)=XC$ for $C\in M_{mp}$ and $\varphi(D)=XD$ for $D\in M_{mq}$.
\end{lemma}

The proof (omitted) is similar to that of  Lemma \ref{thm:block-tran-1}.

\begin{lemma}\label{thm:AB=BC}
Suppose $\F$ is an arbitrary field. If $X\in M_{m}$ and $Y\in M_{n}$ satisfy that $XA=AY$ for all $A\in M_{mn}$, then $X=\lambda I_{m}$ and $Y=\lambda I_n$ for certain $\lambda\in\F$.
\end{lemma}

\begin{proof}
For any $(i,j)\in [m]\times [n]$,
$$
XE_{ij}=E_{ij}Y.
$$
Comparing the $(i,j)$ entry, we get $x_{ii}=y_{jj}$.  Comparing the $(p,j)$ entry for $p\ne i$, we get
$x_{pi}=0$. Comparing the $(i,q)$ entry for $q\ne j$, we get $0=y_{jq}$.
Therefore, $X=\lambda I_{m}$ and $Y=\lambda I_n$ for some $\lambda \in\F$.
\end{proof}

In the remaining of  this section, we assume that ${\rm char} (\F)\ne 2$,  $\L$ is a DUT ladder, and $t:=|\gamma_{\L}|+1$. 
Next we present several results on the image of a derivation of $M_{\L}$. 

\begin{lemma}\label{thm:I_kk}
For $f\in\Der(M_{\L})$ and $(k,k)\in[I(\L)]$, the $f$-images of the identity matrix and the standard matrices 
in the $(k,k)$ block satisfy that 
\begin{equation}\label{I_kk-1t-1}
f(\wt{I_{kk}}),\  f(\wt{E^{[kk]}_{\ell\ell}})\ \in\ \sum_{i=1}^{k-1}\wt{\M_{ik}}+\sum_{j=k+1}^{t} \wt{\M_{kj}} +\pr{\cen{M_{\L}}\cap M_{\L}}
\end{equation}
where (by Lemma \ref{thm:centralizer})
\begin{equation}\label{I_kk-1t-2}
\cen{M_{\L}}\cap M_{\L}=
\begin{cases}
\F I_n &\text{if $\L=\L_B$;}
\\
\wt{\M_{1t}}&\text{if $(1,1)\not\in [I(\L)]$ and $(t,t)\not\in [I(\L)]$;}
\\
0 &\text{otherwise.}
\end{cases}
\end{equation}
\end{lemma}

\begin{proof}
We prove \eqref{I_kk-1t-1} for $f(\wt{I_{kk}})$ here, and the case of $ f(\wt{E^{[kk]}_{\ell\ell}})$ is similar. 
\begin{enumerate}
 \item First we investigate $f(\wt{I_{kk}})_{jj}$.
When $k<j$,
\begin{eqnarray*}
f(\wt{A_{kj}})_{kj}
	&=& f([\wt{I_{kk}},\wt{A_{kj}}])_{kj}=[f(\wt{I_{kk}}),\wt{A_{kj}}]_{kj}+[\wt{I_{kk}},f(\wt{A_{kj}})]_{kj}
\\
	&=& f(\wt{I_{kk}})_{kk}A_{kj}-A_{kj}f(\wt{I_{kk}})_{jj}+f(\wt{A_{kj}})_{kj}
\end{eqnarray*}
Therefore
$$
f(\wt{I_{kk}})_{kk}A_{kj}=A_{kj}f(\wt{I_{kk}})_{jj}\qquad\text{for}\quad A_{kj}\in\M_{kj}.
$$
Lemma \ref{thm:AB=BC} implies that
$f(\wt{I_{kk}})_{kk}=\lambda I_{kk}$ and $f(\wt{I_{kk}})_{jj}=\lambda I_{jj}$ for a $\lambda\in \F$.
The same equation holds for $k>j$.
In the situation $\L\ne\L_B$,  there exists  $(p,p)\not\in[I(\L)]$, which forces $f(\wt{I_{kk}})_{pp}=0$ and thus
 $f(\wt{I_{kk}})_{jj}=0$ for all $j\in[t]$.

\item
Next we prove that  $f(\wt{I_{kk}})_{ij}=0$ for $i<j$,  $i\ne k$, $j\ne k$, and $(i,j)\ne (1,t)$.
Either $i>1$ or $j<t$. Without loss of generality, suppose $j<t$ (similarly for $i>1$).  Then
\begin{equation}\label{I_kk-A_jt}
f([\wt{I_{kk}}, \wt{A_{jt}}])_{it}=[f(\wt{I_{kk}}), \wt{A_{jt}}]_{it}+[\wt{I_{kk}}, f(\wt{A_{jt}})]_{it}.
\end{equation}
\begin{enumerate}
\item
If $k\ne t$, then \eqref{I_kk-A_jt} becomes
$0=f(\wt{I_{kk}})_{ij} A_{jt}$ for any $ A_{jt}\in\M_{jt}$.  So $f(\wt{I_{kk}})_{ij}=0$.
\item
If $k=t$, then  \eqref{I_kk-A_jt} becomes
$$
-f(\wt{A_{jt}})_{it}
=f(\wt{I_{kk}})_{ij} A_{jt}-f(\wt{A_{jt}})_{it}.
$$
Again we get $0=f(\wt{I_{kk}})_{ij} A_{jt}$ and thus $f(\wt{I_{kk}})_{ij}=0$.
\end{enumerate}

\item
Finally,
if $(1,1)\in[I(\L)]$ or $(t,t)\in[I(\L)]$, say $(1,1)\in [I(\L)]$, then for any $(k,k)\in [I(\L)]$ and $k\not\in\{1,t\}$,
$$
0=f([\wt{I_{11}},\wt{I_{kk}}])_{1t}=[f(\wt{I_{11}}),\wt{I_{kk}}]_{1t}+[\wt{I_{11}},f(\wt{I_{kk}})]_{1t}=f(\wt{I_{kk}})_{1t}.
$$
\end{enumerate}
Lemma \ref{thm:centralizer} implies   \eqref{I_kk-1t-2}.  Therefore, \eqref{I_kk-1t-1} is proved. \qedhere
\end{proof}

 For $(p,q)\in [I(\L)]$, we have
$$
\wt{\M_{pq}}\cap [M_{\L}, M_{\L}]=
\begin{cases}
\wt{\s\l_{pp}}, &\text{if $p=q$;}
\\
\wt{\M_{pq}}, &\text{if $p<q$.}
\end{cases}
$$
Next we investigate the image of derivations on each block in $[M_{\L},M_{\L}]$.

\begin{lemma} \label{thm:A_pq}
Suppose ${\rm char}(\F)\ne 2$. 
For $f\in\Der(M_{\L})$, $(p,q)\in [I(\L)]$, and $\wt{A_{pq}}\in \wt{\M_{pq}}\cap [M_{\L}, M_{\L}]$,
\begin{equation}\label{f(A_pq)}
f(\wt{A_{pq}})\in \wt{\M_{pq}}+ \sum_{i=1}^{p-1}\wt{\M_{iq}}+\sum_{j=q+1}^{t}\wt{\M_{pj}}.
\end{equation}
\end{lemma}

\begin{proof}
There are two cases for $(p,q)\in [I(\L)]$:
\begin{enumerate}
\item $p=q$: Then $ \wt{\M_{pq}}\cap [M_{\L}, M_{\L}]=\wt{\s\l_{pp}}=[\wt{\s\l_{pp}},\wt{\s\l_{pp}}]$. For $B_{pp}, C_{pp}\in\s\l_{pp}$,
\begin{equation}\label{diag-block-image}
f([\wt{B_{pp}},\wt{C_{pp}}])=[f(\wt{B_{pp}}),\wt{C_{pp}}]+[\wt{B_{pp}},f(\wt{C_{pp}})].
\end{equation}
Since $f(\wt{B_{pp}})$ and $f(\wt{C_{pp}})$ are block upper triangular matrices with respect to $\gamma_{\L}$, the nonzero $(i,j)$ blocks of 
the right side of \eqref{diag-block-image} satisfy that $p=i\le j$ or $i\le j=p$. 
Thus \eqref{f(A_pq)} holds in this case.

\item $p<q$: Then $ \wt{\M_{pq}}\cap [M_{\L}, M_{\L}]=\wt{\M_{pq}}$.  Let $q=p+k$ and we prove \eqref{f(A_pq)} by induction on $k$.
For better display, we also use $\{\cdot\}_{ij}$ here to denote the embedding of $\M_{ij}$ to $\wt{\M_{ij}}\subseteq M_n$.
\begin{enumerate}
\item  $k=1$:
By Theorem \ref{thm:dominant-ladder}, at least one of $(p,p)$ and $(p+1,p+1)$ is in $[I(\L)]$.  Without lost of generality, suppose
$(p,p)\in [I(\L)]$.
Then for $A_{p,p+1}\in \M_{p,p+1}$,
\begin{eqnarray*}
 f(\wt{A_{p,p+1}})
 &=& f([\wt{I_{pp}},\wt{A_{p,p+1}}])
\\
	&=& [f(\wt{I_{pp}}),\wt{A_{p,p+1}}]+[\wt{I_{pp}},f(\wt{A_{p,p+1}})]
\\
	& 
	{=}&
	\sum_{i=1}^{p-1}
	\vpr{ f(\wt{I_{pp}})_{ip}A_{p,p+1} }_{i,p+1}+\sum_{j=p+1}^{t} \vpr{ f(\wt{A_{p,p+1}})_{pj} }_{pj} -\sum_{i=1}^{p-1}  \vpr{ f(\wt{A_{p,p+1}})_{ip} }_{ip}
\end{eqnarray*}
where the last equality is given by Lemma \ref{thm:I_kk}.
Therefore,
$$
f(\wt{A_{p,p+1}}) +\sum_{i=1}^{p-1}  \vpr{ f(\wt{A_{p,p+1}})_{ip} }_{ip}
=\sum_{i=1}^{p-1}
	\vpr{ f(\wt{I_{pp}})_{ip}A_{p,p+1} }_{i,p+1}+\sum_{j=p+1}^{t} \vpr{ f(\wt{A_{p,p+1}})_{pj} }_{pj}
$$
One one hand, as ${\rm char}(\F)\ne 2$, the nonzero blocks on the left side of the above equality are those of $f(\wt{A_{p,p+1}})$; 
on the other hand,  the right side of this equality has nonzero $(i,j)$ blocks only  for 
$1\le i\le p-1<p+1=j$ or $i=p<p+1\le j\le t$. 
So $k=1$ is done.
\item
$k=\ell$: Suppose the statement is true for all $k<\ell $ where $\ell\ge 2$.
Now $\wt{M_{p,p+\ell}}=[\wt{M_{p,p+1}},\wt{M_{p+1,p+\ell}}],
$
and
\begin{equation*}
  f([\wt{B_{p,p+1}},\wt{C_{p+1,p+\ell}}])
  =
  [f(\wt{B_{p,p+1}}),\wt{C_{p+1,p+\ell}}]+[\wt{B_{p,p+1}},f(\wt{C_{p+1,p+\ell}})]
\end{equation*}
By induction hypothesis,  $f(\wt{B_{p,p+1}})$ has nonzero blocks only on the $p$ block row and the $(p+1)$ block column,
so that the nonzero blocks of $[f(\wt{B_{p,p+1}}),\wt{C_{p+1,p+\ell}}]$ only locate on the $p$ block row and the $(p+\ell)$ block row.
Similarly for $[\wt{B_{p,p+1}},f(\wt{C_{p+1,p+\ell}})]$. 
So \eqref{f(A_pq)}  is true for $k=\ell$.
\item
Overall,  \eqref{f(A_pq)} is verified for all the cases. \qedhere
\end{enumerate}

\end{enumerate}

\end{proof}

Now we are ready to prove  Theorem \ref{Main theorem}.
The basic idea is to explore what remain in $\Der(M_{\L})$ after factoring out $\ad \left(\nor{M_{\L}}\right)|_{M_{\L}}=\ad \left(M_{\L_B}\right)|_{M_{\L}}$.
Given $X\in M_{\L_B}$, $A\in M_{\L}$,
\begin{equation*}
\ad X(A) 
=\sum_{1\le p\le q\le t} \sum_{(i,j)\in [I(\L)]} [\wt{X_{pq}}, \wt{A_{ij}}].
\end{equation*}
A summand $[\wt{X_{pq}}, \wt{A_{ij}}]$ is nonzero only if  $i=q$ or $p=j$. In other words,
$\ad \wt{X_{pq}}$ has nonzero action only on  the $q$ block row or the $p$ block column of $A$.
It motivates us to investigate the relationship of $f(\wt{A_{ip}})$ and $f(\wt{A_{qj}})$ for given $f\in\Der(M_{\L})$ and  $1\le p\le q\le t$.

\begin{proof}[Proof of Theorem \ref{Main theorem}]
\text{}

\begin{enumerate}
\item
If  $M_\L$ is a set of block upper triangular matrices (i.e. ${\L}={\L_B}$),  by \cite[Theorem 4.1]{WQ} and the assumption ${\rm char}(\F)\ne 2$, every $f\in \Der(M_{\L})$ corresponds to
 $X\in M_{\L}$ and $\mu\in {M_{\L}}^*$ such that
$$
f(A)=\ad X(A)+\mu(A) I_n.
$$
Then $\mu([M_{\L},M_{\L}])=0$ by derivation property.  All $\wt{\M_{ij}}$ with $i<j$ are in $[M_{\L}, M_{\L}]$. So  $\mu(A)=\sum_{k\in[t]}\mu (\wt{A_{kk}}).$ Recall that   the $(p,q)$ standard matrix in  $\M_{ij}$ is denoted by $E_{pq}^{[ij]}$. 
Given $k\in [t]$, we have $\wt{A_{kk}}-\tr(A_{kk})\wt{E_{11}^{[kk]}}\in [M_{\L}, M_{\L}]$ so that
$$
\mu\pr{\wt{A_{kk}}}=\tr(A_{kk})\mu\pr{\wt{E_{11}^{[kk]}}}.
$$
Denote  $c_k=\mu\pr{\wt{E_{11}^{[kk]}}}$.  Then
$$
f(A)=\ad X(A)+\left(\sum_{k\in[t]} c_k\tr(A_{kk}) \right) I_n.
$$
This is  \eqref{Der-block-upper-triagular}.  The formulae \eqref{Der-M_L} and \eqref{Der-M_L-2} for $\L=\L_B$ immediately follow.

\item
{\em In the remaining of the proof, we assume   $\L\ne \L_B$, so that $M_{\L}$ has at least one zero diagonal block with respect to the partition $\gamma_{\L}$.}

Suppose $(k,k)\in [I(\L)]$.  For any $A_{kk},B_{kk}\in\M_{kk}$,
\begin{eqnarray*}
f([\wt{A_{kk}}, \wt{B_{kk}}])_{kk}
=
 [f(\wt{A_{kk}})_{kk}, B_{kk}]+[A_{kk}, f(\wt{B_{kk}})_{kk}].
\end{eqnarray*}
So $f(\wt{\cdot})_{kk}:\M_{kk}\to\M_{kk}$ is a derivation of $\M_{kk}$.  Since ${\rm char}(\F)\ne 2$,
according to \cite[Corollary 5.1]{WQ} \footnote{$\Der\left(\gl(m,\F)\right)$ has additional elements when $\rm{char}(\F)=2$ and $m=2$ \cite[Corollary 5.1]{WQ}.},
there is $X_{kk}\in\M_{kk}$ and $\lambda_k\in\F$ such that
$$f(\wt{A_{kk}})_{kk}=[X_{kk}, A_{kk}]+\lambda_k  \tr (A_{kk}) I_{kk} \qquad\text{for \ }A_{kk}\in\M_{kk}.$$
We prove that $\lambda_k =0$  for all $k$.
Recall that $E_{pq}^{[ij]}$  denotes the $(p,q)$ standard matrix in  $\M_{ij}$. 
On one hand, the $(1,1)$ entry of
$$f(\wt{E^{[kk]}_{11}})_{kk}=[X_{kk}, E^{[kk]}_{11}]+\lambda_k I_{kk}$$
equals to $\lambda_k$.
On the other hand, for any $\ell\in[t]$ with $\ell> k$, 
     \begin{eqnarray*}
    f(\wt{E_{11}^{[k\ell]}})_{k \ell} &=& f([\wt{E^{[kk]}_{11}}, \wt{E_{11}^{[k\ell]}}])_{k \ell}=[f(\wt{E^{[kk]}_{11}}), \wt{E_{11}^{[k\ell]}}]_{k \ell}+[\wt{E^{[kk]}_{11}}, f(\wt{E_{11}^{[k\ell]}})]_{k \ell}\\
    &=& f(\wt{E^{[kk]}_{11}})_{kk} E_{11}^{[k\ell]} - E_{11}^{[k\ell]} f({\wt{E^{[kk]}_{11}}})_{ \ell \ell}+ E^{[kk]}_{11} f(\wt{E_{11}^{[k\ell]}})_{k \ell}.
    \end{eqnarray*}
Therefore,
$$
 f(\wt{E^{[kk]}_{11}})_{kk} E_{11}^{[k\ell]}= (I_{kk} - E^{[kk]}_{11}) f(\wt{E_{11}^{[k\ell]}})_{k \ell}
 + E_{11}^{[k\ell]} f({\wt{E^{[kk]}_{11}}})_{ \ell \ell} .
$$
Comparing the $(1,1)$ entry of both sides, we see that
 the $(1,1)$ entries of $f(\wt{E^{[kk]}_{11}})_{kk}$  and $f(\wt{E^{[kk]}_{11}})_{\ell\ell}$ are equal.
 The same result holds for $\ell<k$.
By assumption $\L \neq \L_{B}$.  So there exists $(\ell,\ell)\not \in [I(\L)]$, where $f(\wt{E^{[kk]}_{11}})_{\ell\ell}=0$.  Hence $\lambda_k=0$. Overall, for any $(k,k)\in [I(\L)]$, 
there exists $X_{kk}\in\M_{kk}$ such that
$$f(\wt{A_{kk}})_{kk}=[X_{kk}, A_{kk}] \qquad\text{for all \ }A_{kk}\in\M_{kk}.$$

\item  Given $p, q\in [t]$ and $p<q$, we claim that there exists  $X_{pq}\in\M_{pq}$ such that
\begin{eqnarray}
\label{A_ip}
\wt{f(\wt{A_{ip}})_{iq}}=\ad\wt{X_{pq}} (\wt{A_{ip}}),&& \text{for  any \ } (i,p)\in [I(\L)],\quad\text{and}
\\
\label{A_qj}
\wt{f(\wt{A_{qj}})_{pj}}=\ad\wt{X_{pq}}(\wt{A_{qj}}),&& \text{for  any \ } (q,j)\in [I(\L)].
\end{eqnarray}

There are several situations:
\begin{enumerate}

\item Suppose $(q,j)=(t,t)\in [I(\L)]$. For any $A_{tt}, B_{tt}\in \M_{tt}$,
\begin{align*}
f([\wt{A_{tt}}, \wt{B_{tt}}])_{pt}&=[f(\wt{A_{tt}}), \wt{B_{tt}}]_{pt}+[\wt{A_{tt}}, f(\wt{B_{tt}})]_{pt}=f(\wt{A_{tt}})_{pt}B_{tt}-f(\wt{B_{tt}})_{pt}A_{tt}.
\end{align*}
Set $B_{tt}=I_{tt}$.  Then $f(\wt{A_{tt}})_{pt}=f(\wt{I_{tt}})_{pt}A_{tt}$ for $A_{tt}\in \M_{tt}$. 
Denote  $X_{pt}:=f(\wt{I_{tt}})_{pt}\in \M_{pt}$.  We have $f(\wt{A_{tt}})_{pt}=X_{pt}A_{tt}$ and so
$$\wt{f(\wt{A_{tt}})_{pt}}=\ad\wt{X_{pt}}(\wt{A_{tt}})
\qquad\text{for all}\quad
A_{tt}\in \M_{tt}.
$$

\item
Suppose $(i, p)=(1,1)\in[I(\L)]$.  Similarly, let $Y_{1q}:=-f(\wt{I_{11}})_{1q}\in \M_{1q}$ then 
$$
\wt{f(\wt{A_{11}})_{1q}}=-\wt{A_{11}Y_{1q}}=\ad\wt{Y_{1q}}(\wt{A_{11}}) \qquad\text{for all}\quad A_{11}\in \M_{11}.
$$

    \item
Suppose $(q,j)\in [I(\L)]-\{(t,t)\}$.  Either $q<t$ or $j<t$. Without loss of generality, suppose $j<t$.  
Let $j':=j+1$. Then
$(j,j'), (q,j'), (p,j),(p,j')\in [I(\L)]$, and $\wt{\M_{qj'}}=\wt{\M_{qj}}\wt{\M_{jj'}}=[\wt{\M_{qj}},\wt{\M_{jj'}}]$. For any $A_{qj}\in\M_{qj}$, $A_{jj'}\in\M_{jj'}$,
$$
f(\wt{A_{qj} A_{jj'}})_{pj'}
=f([\wt{A_{qj}},\wt{A_{jj'}}])_{pj'}=[f(\wt{A_{qj}}),\wt{A_{jj'}}]_{pj'}+[\wt{A_{qj}},f(\wt{A_{jj'}})]_{pj'}
=f(\wt{A_{qj}})_{pj}A_{jj'}.
$$
Applying Lemma \ref{thm:block-tran-2} to $\phi:\M_{qj'}\to \M_{pj'}$ defined by $\phi(C):=f(\wt{C})_{pj'}$ and
$\varphi:\M_{qj}\to \M_{pj}$ defined by $\varphi(D):=f(\wt{D})_{pj}$, we can find $X_{pq}\in\M_{pq}$ such that
$f(\wt{A_{qj}})_{pj}=X_{pq}A_{qj}$  for all $A_{qj}\in\M_{qj}$, 
 and $f(\wt{A_{qj'}})_{pj'}=X_{pq}A_{qj'}$  for all $A_{qj'}\in\M_{qj'}$. 
In particular, $X_{pq}$ is independent of $j$. So 
$$\wt{f(\wt{A_{qj}})_{pj}}=\ad\wt{X_{pq}}(\wt{A_{qj}})\qquad\text{for all}\quad A_{qj}\in\M_{qj}.$$ 

\item
Suppose $(i,p)\in [I(\L)]-\{(1,1)\}$.
Either $i>1$ or $p>1$. Without loss of generality, suppose $i>1$ (similarly for $p>1$). 
Let $i':=i-1$. 
 Then $(i',i),(i',p),(i,q),(i',q)\in [I(\L)]$, and
$\wt{\M_{i'p}}=\wt{\M_{i'i}}\wt{\M_{ip}}=[\wt{\M_{i'i}},\wt{\M_{ip}}]$.
For $A_{i'i}\in \M_{i'i}$ and $A_{ip}\in\M_{ip}$,
$$
f(\wt{A_{i'i} A_{ip}})_{i'q}
=f([\wt{A_{i'i}},\wt{A_{ip}}])_{i'q}=[f(\wt{A_{i'i}}),\wt{A_{ip}}]_{i'q}+[\wt{A_{i'i}},f(\wt{A_{ip}})]_{i'q}
=A_{i'i}f(\wt{A_{ip}})_{iq}.
$$
Applying  Lemma \ref{thm:block-tran-1} to  $\phi:\M_{i'p}\to\M_{i'q}$ defined by $\phi(C):=f(\wt{C})_{i'q}$ and
$\varphi:\M_{ip}\to\M_{iq}$ defined by $\varphi(D):=f(\wt{D})_{iq}$, we can find $-Y_{pq}\in\M_{pq}$ such that
$f(\wt{A_{ip}})_{iq}=-A_{ip}Y_{pq}$ for all  $A_{ip}\in\M_{ip}$,
and $f(\wt{A_{i'p}})_{i'q}=-A_{i'p}Y_{pq}$ for all  $A_{i'p}\in\M_{i'p}$. So $-Y_{pq}$ is indepedent of $i$ and
$$
\wt{f(\wt{A_{ip}})_{iq}}=\ad \wt{Y_{pq}}(\wt{A_{ip}}),\qquad\text{for all}\quad A_{ip}\in\M_{ip}.
$$
 
\item
Given any $(i,p), (q,j)\in [I(\L)]$, 
we have $[\wt{A_{ip}},\wt{A_{qj}}]=0$, so that
\begin{eqnarray*}
0&=&f([\wt{A_{ip}},\wt{A_{qj}}])_{ij}
=[f(\wt{A_{ip}}),\wt{A_{qj}}]_{ij}+[\wt{A_{ip}},f(\wt{A_{qj}})]_{ij}
\\
&=& f(\wt{A_{ip}})_{iq}A_{qj}+A_{ip}f(\wt{A_{qj}})_{pj}
=-A_{ip}Y_{pq}A_{qj}+A_{ip}X_{pq}A_{qj}.
\end{eqnarray*}
Therefore, $X_{pq}=Y_{pq}$.
\end{enumerate}

Overall, we successfully find $X_{pq}$ that satisfies \eqref{A_ip} and \eqref{A_qj}.

\item
From (2) and (3), we can construct a matrix in $\M_{\L}$:
$$
X_0\ :=\ \sum_{(k,k)\in [I(\L)]} \wt{X_{kk}}+\sum_{1\le p<q\le t} \wt{X_{pq}}.
$$
Define the derivation
\begin{equation}\label{f_1}
f_1 :=f- \ad X_0.
\end{equation}
Then  for any $(k,k)\in [I(\L)]$, $1\le p<q\le t$, and $(i,p), (q,j)\in [I(\L)]$, we have
$$
f_1(\wt{\M_{kk}})_{kk}=0,\qquad
f_1(\wt{\M_{ip}})_{iq}=0,\qquad
f_1(\wt{\M_{qj}})_{pj}=0.
$$
By Lemmas \ref{thm:I_kk} and \ref{thm:A_pq},  $f_1$ belongs to the following set: 
\begin{eqnarray}
D_0 :=\{g\in\Der(M_{\L}) &\mid& g(\wt{\M_{kk}})\in \cen{M_{\L}}\cap M_{\L}\text{ for } (k,k)\in [I(\L)],
\notag
\\
&&g(\wt{\M_{pq}})\subseteq  \wt{\M_{pq}}\text{ for } 1\le p<q\le t\}.
\end{eqnarray}
It remains to describe the subalgebra $D_0$ of $\Der(M_{\L})$. 

\item 
Given $f'\in \Der(M_{\L})$,  $p,q\in [t]$ with $p<q$, and  $k\in [t]$ with $p\le k\le q$, Lemmas \ref{thm:I_kk} and \ref{thm:A_pq} imply that
\begin{eqnarray}
f'(\wt{A_{pk}A_{kq}})_{pq}
&=& f'([\wt{A_{pk}},\wt{A_{kq}}])_{pq}
= [f'(\wt{A_{pk}}),\wt{A_{kq}}]_{pq}+[\wt{A_{pk}},f'(\wt{A_{kq}})]_{pq}
\notag\\
&=& f'(\wt{A_{pk}})_{pk}A_{kq}+A_{pk}f'(\wt{A_{kq}})_{kq}.
\label{ApkAkq}
\end{eqnarray}
This formula will be frequently used  in the following computations.

\item 
We prove the following claim regarding $f_1$ defined in \eqref{f_1}:
there exist $Y_{ii}\in \M_{ii}$ for $i\in [t]$, such that for each  $k\in [t]$,  the derivation
$f_{1}^{(k)} :=\left.\pr{ f_1-\sum_{i=1}^{k} \ad \wt{Y_{ii}}}\right|_{M_{\L}}$  has the images
\begin{equation}\label{f_1(k)}
\begin{cases}
f_{1}^{(k)}(\wt{\M_{qq}})=f_1(\wt{\M_{qq}}), &\text{for}\ (q,q)\in [I(\L)],\ q\le k;
\\
f_{1}^{(k)}(\wt{\M_{pq}})=0, &\text{for}\  (p,q)\in[I(\L)],\ 1\le p<q\le k.
\end{cases}
\end{equation}
Moreover,  $Y_{ii} \in\F I_{ii}$ whenever $(i,i)\in [I(\L)]$.

The proof is proceeded  by induction on $k$: 
\begin{enumerate}
 \item $k=1$ and $2$:  There are two subcases:
 \begin{itemize}
 \item
 If $(1,1)\in [I(\L)]$, we let $Y_{11}=0\in\M_{11}$ so that $f_1^{(1)}=f_1$. By \eqref{ApkAkq},
  \begin{eqnarray*}
 f_1^{(1)}(\wt{A_{11}A_{12}})_{12}
 =
 f_1^{(1)}(\wt{A_{11}})_{11} A_{12}+A_{11}f_1^{(1)}(\wt{A_{12}})_{12}
=
 A_{11}f_1^{(1)}(\wt{A_{12}})_{12}.
\end{eqnarray*}
 By Lemma \ref{thm:block-tran-1}, there exists $-Y_{22}\in\M_{22}$, such that $f_1^{(1)}(\wt{A_{12}})_{12}=-A_{12}Y_{22}$.
 Let $f_1^{(2)}= f_1^{(1)}-\ad\wt{Y_{22}}$.  Then $ f_1^{(2)}(\wt{A_{12}})=0$.
If furthermore $(2,2)\in [I(\L)]$, then by \eqref{ApkAkq},
 \begin{eqnarray*}
0=f_1^{(2)}(\wt{A_{12}A_{22}})_{12}=f_1^{(2)}(\wt{A_{12}})_{12}A_{22}+A_{12}f_1^{(2)}(\wt{A_{22}})_{22}=A_{12}f_1^{(2)}(\wt{A_{22}})_{22}.
 \end{eqnarray*}
Thus
$$0=f_1^{(2)}(\wt{A_{22}})_{22}=f_1(\wt{A_{22}})_{22}-[Y_{22},A_{22}]=-[Y_{22},A_{22}].$$
So $Y_{22}\in\F I_{22}$ and $f_1^{(2)}(\wt{A_{22}})=f_1(\wt{A_{22}})$. The claim holds for $k=1, 2$.

  \item
If $(1,1)\not\in [I(\L)]$, then $(2,2)\in [I(\L)]$ by Theorem \ref{thm:dominant-ladder}.
By \eqref{ApkAkq},
$$
f_1(\wt{A_{12}A_{22}})_{12}=f_1(\wt{A_{12}})_{12}A_{22}+A_{12}f_1(\wt{A_{22}})_{22}=f_1(\wt{A_{12}})_{12}A_{22}.
$$
By Lemma \ref{thm:block-tran-2}, there exists $Y_{11}\in\M_{11}$ such that $f_1(\wt{A_{12}})_{12}=Y_{11}A_{12}$.  Let
$Y_{22}=0\in\M_{22}$, $f_1^{(1)}=f_1-\ad\wt{Y_{11}}$, and $f_1^{(2)}=f_1^{(1)}-\ad\wt{Y_{22}}$.  Then the claim holds for $k=1,2$.
\end{itemize}

 \item $k=\ell >2$: Suppose the claim holds for $k=\ell-1\ge 2$.  So there exist $Y_{11}\in\M_{11},\cdots, Y_{\ell-1,\ell-1}\in\M_{\ell-1,\ell-1}$, such that
 $f_{1}^{(\ell-1)} := f_1-\sum_{i=1}^{\ell-1} \ad \wt{Y_{ii}}$ satisfies \eqref{f_1(k)}   for $k=\ell-1$.
 Clearly  $f_{1}^{(\ell-1)} \in D_0$.
For any $p\in[\ell-2]$, by \eqref{ApkAkq},
\begin{eqnarray*}
  f_{1}^{(\ell-1)}(\wt{A_{p,\ell-1}A_{\ell-1,\ell}})_{p\ell}
&=& f_{1}^{(\ell-1)}(\wt{A_{p,\ell-1}})_{p,\ell-1}A_{\ell-1,\ell}+ A_{p,\ell-1}f_{1}^{(\ell-1)}(\wt{A_{\ell-1,\ell}})_{\ell-1,\ell}
\\
&=&  A_{p,\ell-1}f_{1}^{(\ell-1)}(\wt{A_{\ell-1,\ell}})_{\ell-1,\ell}.
\end{eqnarray*}
  By Lemma \ref{thm:block-tran-1}, there exists $-Y_{\ell\ell}\in\M_{\ell\ell}$, such that
  $$
   f_{1}^{(\ell-1)}(\wt{A_{p\ell}})_{p\ell}= -A_{p\ell}Y_{\ell\ell}\qquad\text{for all}\quad p\in[\ell-1].
  $$
  Let
  $$
  f_{1}^{(\ell)}:=f_{1}^{(\ell-1)}-\ad\wt{Y_{\ell\ell}}.
  $$
  Then $ f_{1}^{(\ell)}(\wt{A_{p\ell}})=0$ for $p\in[\ell-1]$.
  In the case $(\ell,\ell)\in[I(\L)]$, by \eqref{ApkAkq},
\begin{eqnarray*}
  0= f_{1}^{(\ell)}(\wt{A_{\ell-1,\ell}A_{\ell\ell}})_{\ell-1,\ell}
  = f_{1}^{(\ell)}(\wt{A_{\ell-1,\ell}})_{\ell-1,\ell}A_{\ell\ell}+
  A_{\ell-1,\ell} f_{1}^{(\ell)}(\wt{A_{\ell\ell}})_{\ell\ell}
= A_{\ell-1,\ell} f_{1}^{(\ell)}(\wt{A_{\ell\ell}})_{\ell\ell}.
\end{eqnarray*}
So
$$
0=f_{1}^{(\ell)}(\wt{A_{\ell\ell}})_{\ell\ell}=\pr{f_1-\sum_{i=1}^{\ell} \ad \wt{Y_{ii}} }(\wt{A_{\ell\ell}})_{\ell\ell}=-[Y_{\ell\ell},A_{\ell\ell}].
$$
Thus $Y_{\ell\ell}\in\F I_{\ell\ell}$ and $ f_{1}^{(\ell)}(\wt{A_{\ell\ell}})= f_{1}(\wt{A_{\ell\ell}})$.
The claim is proved for $k=\ell$.

\item
Overall,  the claim holds for every $k\in [t]$.

\end{enumerate}

\item
The derivation
$f_{1}^{(t)} = f_1-\sum_{i=1}^{t} \ad \wt{Y_{ii}}$
sends each $\wt{\M_{kk}}$ for $(k,k)\in [I(\L)]$ to $\cen{M_{\L}}\cap M_{\L}$, and $\wt{\M_{pq}}$ for $1\le p<q\le t$ to $0$.
For any $A,B\in M_{\L}$,
$$
f_{1}^{(t)}([A,B])=[f_{1}^{(t)}(A),B]+[A,f_{1}^{(t)}(B)]=0.
$$
Therefore, $f_{1}^{(t)}\in \mathcal{D}$ for $\mathcal{D}$ defined in \eqref{Der-D}.
Every $\phi\in{\mathcal D}$ satisfies $\phi([A,B])=0=[\phi(A),B]+[A,\phi(B)]$ for $A,B\in M_{\L}$.  Thus ${\mathcal D}\subseteq\Der{M_{\L}}$. So far we have
$$
\Der(M_{\L})=\left(\ad M_{\L_B}\right)|_{M_{\L}}+ {\mathcal D}.
$$

If $(1,1)\in [I(\L)]$ or $(t,t)\in[I(\L)]$, then $\cen{M_{\L}}\cap M_{\L}=0$ implies that $ {\mathcal D}=0$. We get
\eqref{Der-with-first-or-last}.

If neither $(1,1)$ nor $(t,t)$ is in $[I(\L)]$, then $\cen{M_{\L}}\cap M_{\L}=\wt{\M_{1t}}$. The set $\{\wt{E_{11}^{[kk]}}\mid (k,k)\in [I(\L)]\}$
spans a subalgebra complement to $[M_{\L},M_{\L}]$ in $M_{\L}$.  Then
for any $\phi\in {\mathcal D}$ and $A\in M_{\L}$, 
$$\phi(A)=\sum_{(k,k)\in [I(\L)]} \phi(\wt{A_{kk}})
=\sum_{(k,k)\in [I(\L)]} \tr(A_{kk})\phi(\wt{E_{11}^{[kk]}}).
$$
Denote $\wt{Y_{1tk}}:=\phi(\wt{E_{11}^{[kk]}})\in\wt{\M_{1t}}$ for $(k,k)\in [I(\L)]$.  We get
\eqref{Der-no-first-and-last}.

In all the cases, the equations \eqref{Der-with-first-or-last} and \eqref{Der-no-first-and-last} as well as \eqref{Der-block-upper-triagular} derived in (1) imply \eqref{Der-M_L} and \eqref{Der-M_L-2} by direct verification.
So Theorem \ref{Main theorem} is completely proved.  \qedhere

\end{enumerate}

\end{proof}

\section{Derivations of  $[M_{\L},M_{\L}]$ for SDUT ladder $\L$}

In this section, we will give an explicit description of the derivation algebra of  $[M_{\L},M_{\L}]$ for a  SDUT ladder $\L$  when ${\rm char}(\F)\ne 2, 3$.  
The Lie subalgebra $[M_{\L},M_{\L}]$  consists of matrices in $M_{\L}$ that have zero trace on every diagonal block of $M_{\L}$. 
To see the motivations of studying $\Der([M_{\L},M_{\L}])$, we  make the following notation. 

\begin{df} Given an upper triangular ladder $\L$,
let $M_{\L}^0$ denote the Lie subalgebra of $M_{\L}$ consisting of matrices with zero trace on every diagonal block of $M_{\L}$ with respect to the partition of $\L$.
\end{df}

Any derivation of a Lie algebra $\g$ preserves the lower central series, upper central series, and  derived series of $\g$.
 Given an  upper triangular ladder $\L$,  the derived series of $M_{\L}$ is:
$$
M_{\L}=M_{\L}^{(0)}\unrhd M_{\L}^{(1)}\unrhd M_{\L}^{(2)}\unrhd \cdots ,\qquad M_{\L}^{(k)}:=[M_{\L}^{(k-1)},M_{\L}^{(k-1)}].
$$
The following observations are straighforward in the view point of partitioned matrices:
\begin{enumerate}
\item
When $k\ge 1$, each $M_{\L}^{(k)}=M_{\L_k}^0$ for some upper triangular ladder $\L_k$ contained in $\L$, that is, 
$$
I(\L)\supseteq I(\L_1)\supseteq I(\L_2)\supseteq\cdots
$$
\item
The Lie algebra $M_{\L}$ is non-solvable if and only if its derived series  terminates at a nonzero $M_{\L_*}^0$, where $\L_*$ is the maximal SDUT ladder contained in ${\L}$.
Precisely,  
$$\L_*=\{(i_{\ell},j_{\ell})\in\L\mid i_{\ell}>j_{\ell}\}.$$
If $\L_*$ given above is an empty set, then  $M_{\L}$ is solvable, and  the derived series of $M_{\L}$  terminates at $0$.
\item
Every $f\in\Der(M_{\L})$ stabilizes $M_{\L_*}^0$ and induces a derivation $f|_{M_{\L_*}^0}\in\Der(M_{\L_*}^0)$. 
The restriction map $\pi:\Der(M_{\L})\to \Der(M_{\L_*}^0)$ is a Lie algebra homomorphism. 

\end{enumerate}

\begin{ex}
In $M_8$,  the forms of $M_{\L}$, $M_{\L_*}$, and $M_{\L^*}^0$ associate with an upper triangular $\L$ are illustrated below.
In particular, we see that $\L_*=\{(2,1),(7,6)\}$ is  SDUT.
$$
{\small
\begin{array}{c|c|c}
M_{\L}	&M_{\L_*}	&M_{\L_*}^0
\\ \hline
	&	&
\\
\left(\begin{array}{cc:c:c:c:cc:c}
*	&*	&*	&*	&*	&*	&*	&*
\\
*	&*	&*	&*	&*	&*	&*	&*
\\ \hdashline
	&	&*	&*	&*	&*	&*	&*
\\ \hdashline
	&	&	&*	&*	&*	&*	&*
\\ \hdashline
	&	&	&	&0	&*	&*	&*
\\ \hdashline
	&	&	&	&	&*	&*	&*
\\
	&	&	&	&	&*	&*	&*
\\  \hdashline
	&	&	&	&	&	&	&*
\end{array}\right)
	&
\left(\begin{array}{cc:c:c:c:cc:c}
*	&*	&*	&*	&*	&*	&*	&*
\\
*	&*	&*	&*	&*	&*	&*	&*
\\ \hdashline
	&	&0	&0	&0	&*	&*	&*
\\ \hdashline
	&	&	&0	&0	&*	&*	&*
\\ \hdashline
	&	&	&	&0	&*	&*	&*
\\ \hdashline
	&	&	&	&	&*	&*	&*
\\
	&	&	&	&	&*	&*	&*
\\  \hdashline
	&	&	&	&	&	&	&0
\end{array}\right)
	&
\left(\begin{array}{cc:c:c:c:cc:c}
a	&*	&*	&*	&*	&*	&*	&*
\\
*	&-a	&*	&*	&*	&*	&*	&*
\\ \hdashline
	&	&0	&0	&0	&*	&*	&*
\\ \hdashline
	&	&	&0	&0	&*	&*	&*
\\ \hdashline
	&	&	&	&0	&*	&*	&*
\\ \hdashline
	&	&	&	&	&b	&*	&*
\\
	&	&	&	&	&*	&-b	&*
\\  \hdashline
	&	&	&	&	&	&	&0
\end{array}\right)
\end{array}
}
$$
\end{ex}

The above observations indicate that the structure of  $\Der(M_{\L}^0)$ for  SDUT ladders $\L$ (where $\L=\L_*$) will be useful in studying the structure of $\Der(M_{\L'})$ for non-solvable upper triangular ladders $\L'$.
{\em In the rest of this section, we assume that $\L$ is a SDUT ladder, unless otherwise  specified.}
Let $t:=|\gamma_{\L}|+1$ as before.  

\begin{theorem}\label{thm:strongly dominant}
Suppose ${\rm char}(\F)\ne 2, 3$.  Let $\L$ be a SDUT ladder of size $n$.  Then
every derivation $f\in\Der(M_{\L}^0)$ can be extended to a derivation $f^+\in\Der(M_{\L})$ such that
$f^+|_{M_{\L}^0}=f$.  In particular, there exists a block upper triangular matrix $X\in M_{\L_B}$ such that
\begin{equation}\label{SDUT-Der-1}
f(B)=\ad X(B)=[X, B],\qquad\text{for all \ } B\in M_{\L}^0.
\end{equation}
We can write
\begin{equation}\label{SDUT-Der-2}
\Der(M_{\L}^0)=\ad(\nor{M_{\L}}/\cen{M_{\L}})|_{M_{\L}^0}.
\end{equation}
\end{theorem}

The proof of Theorem \ref{thm:strongly dominant} will be defered to the end of this section.

\begin{cor}
When ${\rm char}(\F)\ne 2, 3$, and $\L$ is a SDUT ladder, we have the split exact sequence:
\begin{equation}\label{SDUT-exact}
0\hookrightarrow {\mathcal D}\hookrightarrow \Der(M_{\L})\overset{\pi}{\twoheadrightarrow}  \Der(M_{\L}^0)\twoheadrightarrow  0,
\end{equation}
where ${\mathcal D}$ is defined in \eqref{Der-D}.
\end{cor}

\begin{proof}
Theorem \ref{thm:strongly dominant} shows that the restriction map $\pi:\Der(M_{\L})\to \Der(M_{\L}^0)$ is surjective.
Theorem \ref{Main theorem} shows that $\Der(M_{\L})  = \ad(\nor{M_{\L}}/\cen{M_{\L}})|_{M_{\L}}\oplus\mathcal{D}$. 
It is easy to check that $\cen{M_{\L}}=\cen{M_{\L}^0}$ and  $\Ker\pi={\mathcal D}$. Therefore, we get the split exact sequence \eqref{SDUT-exact}. 
\end{proof}

\begin{ex}
When ${\rm char}(\F)=2$ or $3$, we show by  counterexamples that $\Der(M_{\L}^0)$ is not in the form of \eqref{SDUT-Der-2}. 

\begin{itemize}
\item   ${\rm char}(\F)=2$: 
Let $M_{\L}=M_2$, so that $M_{\L}^{0}=\s\l_2$. 
Let $f$ be the derivation of $M_2$ given in Example \ref{ex-main-char-2}, that is, $f(E_{12})=E_{21}$, and $f(E_{ij})=0$ for $(i,j)\in\{(1,1), (2,2), (2,1)\}$.  Then $f|_{\s\l_2}$ is a derivation of $\s\l_2$.  However, there is no $X\in M_{\L_B}=M_2$ such that
$f|_{\s\l_2}(E_{12})=[X,E_{12}]$. 

\item  ${\rm char}(\F)=3$:
Let $n=4$, $\L=\{(2,1)\}$. Then $M_{\L}^0$ consists of matrices in $M_4$ that takes the following forms:
$$
\mtx{a_{11} &a_{12} &a_{13} &a_{14}\\
a_{21} &-a_{11} &a_{23} &a_{24}\\
0 &0 &0 &0\\
0 &0 &0 &0
},\qquad a_{ij}\in\F.
$$
So $M_{\L}^0$ has a basis ${\mathcal B}=\{E_{11}-E_{22}, E_{12}, E_{13}, E_{14}, E_{21}, E_{23}, E_{24}\}$. Define
$f\in  \End(M_{\L}^0)$ by $f(E_{12}):=E_{24}$, and
$f(E)=0$ for all other matrices $E$ in the basis ${\mathcal B}$.
We prove that
\begin{equation}\label{char-3-der}
f([E,E'])=[f(E),E']+[E,f(E')]
\end{equation}
for any distinct $E,E'\in{\mathcal B}$, so that $f\in \Der(M_{\L}^0)$.  The only case that the left side or the right side of \eqref{char-3-der} is nonzero is
$\{E, E'\}=\{E_{11}-E_{22}, E_{12}\}$, in which
$$f([E,E'])=2f(E_{12})=2E_{24},\qquad  [f(E),E']+[E,f(E')]=-E_{24}.
$$
Since ${\rm char}(\F)=3$, the equality \eqref{char-3-der} holds for this case.
Therefore, \eqref{char-3-der} holds  for all $\{E,E'\}\subseteq {\mathcal B}$, and $f\in\Der(M_{\L}^0)$.
However,  there is no matrix $X\in M_4$, such that $f(E_{12})=[X, E_{12}]$.
\end{itemize}
\end{ex}

In order to prove Theorem \ref{thm:strongly dominant}, 
 we first give two lemmas similar to Lemmas \ref{thm:block-tran-1} and \ref{thm:block-tran-2}.

\begin{lemma}\label{thm:block-tran-3}
Suppose $n\ge 2$.
 If linear transformations $\phi:M_{mn}\to M_{mq}$ and $\varphi: \s\l_{n}\to M_{nq}$ satisfy that
\begin{equation}\label{block-tran-eq1}
\phi(AB)=A\varphi(B)\qquad\text{for all}\quad A\in M_{m n},\ \ B\in \s\l_{n},
\end{equation}
then  there is $X\in M_{nq}$ such that
$\phi(C)=CX$ for $C\in M_{mn}$ and $\varphi(D)=DX$ for $D\in \s\l_{n}$.
\end{lemma}

Lemma \ref{thm:block-tran-3} is very similar  to a special case ($p=n$) of Lemma \ref{thm:block-tran-1}, except that the domain of $\varphi$ is $\s\l_{n}$ instead of $M_{nn}=M_{n}$.  The proof of  Lemma \ref{thm:block-tran-3} (omitted) is totally parallel to 
that of Lemma \ref{thm:block-tran-1}, using the key fact that $\{\E{1j}{mn}B\mid j\in[n], B\in\s\l_n\}$ still spans the first row space of $M_{mn}$.
 Similarly, we have the following lemma.

\begin{lemma}\label{thm:block-tran-4}
Suppose $n\ge 2$.  If linear transformations $\phi:M_{nq}\to M_{mq}$ and $\varphi: \s\l_{n}\to M_{mn}$ satisfy that
\begin{equation}\label{block-tran-eq2}
\phi(BA)=\varphi(B)A\qquad\text{for all}\quad  A\in M_{nq},\ B\in \s\l_{n},
\end{equation}
then  there is $X\in M_{mn}$ such that
$\phi(C)=XC$ for $C\in M_{nq}$ and $\varphi(D)=XD$ for $D\in \s\l_{n}$.
\end{lemma}

Next we give two lemmas related to the bracket operation.

\begin{lemma}\label{thm:sl_n-case}
Suppose ${\rm char}(\F)\ne 2,3$. If a linear transformation $\phi:{\s\l}_{n}\to M_{nm}$ satisfies that
\begin{equation}\label{sl_n-case}
\phi(AB-BA)=A\phi(B)-B\phi(A),\quad  \text{for all}\ A, B\in {\s\l}_{n},
\end{equation}
then there is $X\in M_{nm}$ such that $\phi(C)=CX$ for $C\in {\s\l}_{n}$.
\end{lemma}

\begin{proof}
The case $n=1$ is obviously true.  We now assume that $n\ge 2$.
Let $\{E_{ij}\mid i,j\in[n]\}$ be the standard basis of $M_n$.  
Then $\s\l_n$ has the standard basis
$\{E_{ij}\mid i,j\in[n], i\ne j\}\cup\{H_i\mid i\in[n-1]\}$, where $H_i:=E_{ii}-E_{i+1,i+1}$.
We have $M_n=\s\l_n\oplus\F E_{11}$.

First we prove that the only possibly nonzero row of $\phi(E_{ij})$  ($i\ne j$) is the $i$-th row, and
 the only possibly nonzero rows of $\phi(H_i)=\phi(E_{ii}-E_{i+1,i+1})$ ($i\in [n-1]$) are the $i$-th  and the $(i+1)$-th rows.

Suppose $i,j\in[n]$ with $i<j$.  Denote $E:=E_{ij}$, $F:=E_{ji}$, and $H:=E_{ii}-E_{jj}$.
Then $\{H, E,F\}\in\s\l_n$ is the standard triple of a $\s\l_2$ subalgebra. We have
$$
2\phi(E)=\phi([H,E])=H\phi(E)-E\phi(H)\quad\Longrightarrow\quad (2I_n-H)\phi(E)=-E\phi(H).
$$
When ${\rm char}(\F)\ne 2, 3$, the matrix $2I_n-H=\diag(2,2,\cdots,\underset{i}{1},\cdots,\underset{j}{3},\cdots,2)$ is  invertible and diagonal. The matrix $(2I_n-H)^{-1}$ is again diagonal with 1 as the $i$-th diagonal entry.  So we have
$$
\phi(E)=-(2I_n-H)^{-1}E_{ij}\phi(H)=-E_{ij}\phi(H).
$$
In particular, $\phi(E_{ij})=\phi(E)$ has zeros outside of the $i$-th row.
Similar argument works for $E_{ji}$.

For $H_i=E_{ii}-E_{i+1,i+1}$, we have
$$
\phi(H_i)=\phi([E_{i,i+1},E_{i+1,i}])=E_{i,i+1}\phi(E_{i+1,i})-E_{i+1,i}\phi(E_{i,i+1}).
$$
Therefore,
 $\phi(H_i)$ has zeros outside of the $i$-th  and the $(i+1)$-th rows.

Next we extend the map $\phi$ from the domain $\s\l_n$ to the domain $M_n$ such that property \eqref{sl_n-case} still hold  in $M_n$.  Define the linear transformation $\phi^+:M_n\to M_{nm}$ as follow:
$$
\begin{cases}
 \phi^+(A)=\phi(A), &\text{for $A\in\s\l_n$;}
 \\
 \phi^+(E_{11})=E_{12}\phi(E_{21}).
\end{cases}
$$
Then $\phi^+$ is an extension of $\phi$ from $\s\l_n$ to $M_n$.
To verify \eqref{sl_n-case}-like property for $\phi^+$  in $M_n$, it suffices to  prove the following equality  for   all $A$ in the standard basis of $\s\l_n$:
\begin{equation}\label{E_11-sl_n}
\phi^+(E_{11}A-AE_{11})=E_{11}\phi^+(A)-A\phi^+(E_{11})=E_{11}\phi(A)-A E_{12}\phi(E_{21}).
\end{equation}
\begin{enumerate}
\item $A=E_{1j}$, $1\ne j\in[n]$: the left side of  \eqref{E_11-sl_n} is $\phi^+(E_{1j})=\phi(E_{1j})$.
The right side of  \eqref{E_11-sl_n} is
$E_{11}\phi(E_{1j})$.  Both sides are clearly equal since   $\phi(E_{1j})$ has zero entries outside of the first row.

\item $A=E_{i1}$, $1\ne i\in [n]$: the proof   is similar.

\item $A=E_{ij}$, $i,j\in [n]-\{1\}$, $i\ne j$: both sides of  \eqref{E_11-sl_n} are  zero.

\item $A=H_1=E_{11}-E_{22}$:   the left side of  \eqref{E_11-sl_n} is zero.
The right side of  \eqref{E_11-sl_n} is
$$
E_{11}\phi(H_1)-H_1E_{12}\phi(E_{21})=E_{11}\phi(H_1)-E_{12}\phi(E_{21}).
$$
We have
\begin{eqnarray*}
-2\phi(E_{21})=\phi([H_1,E_{21}])=H_1\phi(E_{21})-E_{21}\phi(H_1)
=-\phi(E_{21})-E_{21}\phi(H_1),
\end{eqnarray*}
where the last equality holds since $\phi(E_{21})$ has zeros outside of the second row.
Therefore, $\phi(E_{21})=E_{21}\phi(H_1)$, and the  right side of  \eqref{E_11-sl_n} is
$$E_{11}\phi(H_1)-E_{12}\phi(E_{21})=E_{11}\phi(H_1)-E_{12}E_{21}\phi(H_1)=0.$$
So both sides are equal.

\item $A=H_i$, $i\in [n-1]-\{1\}$:
 Both sides of  \eqref{E_11-sl_n} are clearly zero.
\end{enumerate}
Overall,  \eqref{E_11-sl_n} is proved. We have
\begin{equation}\label{M_n-case}
\phi^+(AB-BA)=A\phi^+(B)-B\phi^+(A),\quad\text{for all } A, B\in M_n.
\end{equation}

Finally, let $B=I_n$ in \eqref{M_n-case}, then
$$0=A\phi^+(I_n)-I_n\phi^+(A)\quad\Rightarrow\quad \phi^+(A)=A\phi^+(I_n).$$
Setting $X:=\phi^+(I_n)$, we get $\phi(A)=AX$ for  all $A\in\s\l_n$.
\end{proof}

Similarly, we have the following result.

\begin{lemma}\label{thm:sl_n-case-2}
Suppose ${\rm char}(\F)\ne 2,3$. If a linear transformation $\phi:{\s\l}_{n}\to M_{mn}$ satisfies that
\begin{equation}\label{sl_n-case-2}
\phi(AB-BA)=\phi(A)B-\phi(B)A,\quad  \text{for all}\ A, B\in {\s\l}_{n},
\end{equation}
then there is $X\in M_{mn}$ such that $\phi(C)=XC$ for $C\in {\s\l}_{n}$.
\end{lemma}

The statements of Lemmas \ref{thm:sl_n-case} and \ref{thm:sl_n-case-2}  also hold when ${\rm char}(\F)= 2$, but the proofs should be adjusted slightly.
We will not need the case ${\rm char}(\F)= 2$ here.
The following counterexample shows that Lemma \ref{thm:sl_n-case} is not true when ${\rm char}(\F)= 3$.
Likewise for Lemma \ref{thm:sl_n-case-2}.

\begin{ex}
Suppose ${\rm char}(\F)= 3$.  In $M_2$, let $H:=E_{11}-E_{22}$, and $\phi:\s\l_2\to M_2$  the linear map given by
$$
\phi(E_{12}):=E_{21},\qquad
\phi(E_{21}):=0,\qquad
\phi(H):=0.
$$
Then $\phi$ satisfies \eqref{sl_n-case}  since
\begin{eqnarray*}
\phi([H,E_{12}]) &=& 2\phi(E_{12})=2E_{21}=-E_{21}=H\phi(E_{12})-E_{12}\phi(H),
\\
\phi([H,E_{21}])&=& -2\phi(E_{21})=0=H\phi(E_{21})-E_{21}\phi(H),
\\
\phi([E_{12},E_{21}])&=& \phi(H)=0=E_{12}\phi(E_{21})-E_{21}\phi(E_{12}).
\end{eqnarray*}
However, there is no $X\in M_2$ such that $\phi(E_{12})=E_{21}=E_{12}X$.
\end{ex}

\begin{lemma}{\label{range-M_L^0}}
Suppose $\rm {char}(\mathbb{F})\neq 2$.   Then for any $f\in\Der(M_{\L}^0)$: 
\begin{eqnarray}
f(\wt{\s\l_{kk}})
  &\subseteq &
  \wt{\s\l_{kk}}+\sum_{i=1}^{k-1} \wt{\M_{ik}}+\sum_{j=k+1}^{t} \wt{\M_{kj}},\qquad
\text{for}\quad  (k,k)\in [I(\L)];
\label{M_L^0-kk}\\
f(\wt{\M_{pq}})
  &\subseteq &
  \wt{\M_{pq}}+\sum_{i=1}^{p-1} \wt{\M_{iq}}+\sum_{j=q+1}^{t} \wt{\M_{pj}},\qquad
\text{for}\quad  1\le p<q \le t.
\label{M_L^0-pq}
\end{eqnarray}
\end{lemma}

The proof is similar to that of Lemma \ref{thm:A_pq}, with some slight  adjustments.

\begin{proof}
Given  $(k,k)\in [I(\L)]$, we have
$[ \wt{\s\l_{kk}},  \wt{\s\l_{kk}}]= \wt{\s\l_{kk}}$ in $M_{\L}^0$.  For $A_{kk}, B_{kk}\in \s\l_{kk}$,
$$
f([\wt{A_{kk}},\wt{B_{kk}}])=[f(\wt{A_{kk}}),\wt{B_{kk}}]+[\wt{A_{kk}},f(\wt{B_{kk}})]
\ \in\ \wt{\s\l_{kk}}+\sum_{i=1}^{k-1} \wt{\M_{ik}}+\sum_{j=k+1}^{t} \wt{\M_{kj}}.
$$
So \eqref{M_L^0-kk} is done.

Given   $1\le p<q\le t$, we prove \eqref{M_L^0-pq} by induction on $\ell:=q-p$:

\begin{enumerate}
\item $\ell=1$:  Here $(p,q)=(p,p+1)\in [I(\L)]$.  By Theorem \ref{thm:dominant-ladder}, at least one of $(p,p)$ and $(p+1,p+1)$ is in $[I(\L)]$.
Without loss of generality, suppose $(p,p)\in [I(\L)]$.  Since $\L$ is SDUT, the matrices in $\s\l_{pp}$ have the size $m\ge 2$.  Therefore
$[\wt{\s\l_{pp}},\wt{\M_{p,p+1}}]=\wt{\M_{p,p+1}}$ in $M_{\L}^0$.
Let $\{\cdot\}_{ij}$ also denote the embedding of $\M_{ij}$ to $\wt{\M_{ij}}\in M_n$. 
For $A_{pp}\in \s\l_{pp}$, $A_{p,p+1}\in \M_{p,p+1}$,
\begin{eqnarray}\label{A_p(p+1)}
f(\wt{A_{pp} A_{p,p+1}})
&=&
f([\wt{A_{pp}},\wt{A_{p,p+1}}])
=
[f(\wt{A_{pp}}),\wt{A_{p,p+1}}]+[\wt{A_{pp}},f(\wt{A_{p,p+1}})]
\\ \notag
&\in&\wt{\M_{p,p+1}}+\sum_{i=1}^{p-1} \wt{\M_{i,p+1}}+\sum_{j=p+2}^{t} \wt{\M_{pj}}
\\ \notag
&&\qquad
-\sum_{i=1}^{p-1}  \left\{f(\wt{A_{p,p+1}})_{ip}A_{pp}\right\}_{ip} +\left\{[A_{pp}, f(\wt{A_{p,p+1}})_{pp}]\right\}_{pp}.
\end{eqnarray}
 
To get \eqref{M_L^0-pq} for $q-p=1$, 
it remains to prove that $f(\wt{E_{kj}^{[p,p+1]}})_{ip}=0$ for any  given standard matrix $E_{kj}^{[p,p+1]}$ in $\M_{p,p+1}$ and $i\in[p]$. 
  There are two cases:
\begin{itemize}
\item $i\in [p-1]$:  \eqref{A_p(p+1)} shows that for $A_{pp}\in \s\l_{pp}$ and $A_{p,p+1}\in \M_{p,p+1}$,
\begin{align}\label{B1}
f(\wt{A_{pp}A_{p,p+1}})_{ip}&=-f({\wt{A_{p,p+1}}})_{ip}A_{pp}.
\end{align}
Since the size $m$ of $\s\l_{pp}$ is no less than $2$, we can choose $s\in [m]-\{k\}$.  Then
\begin{equation}\label{E_kj}
f(\wt{E^{[p,p+1]}_{kj}})_{ip}
= f(\wt{E^{[pp]}_{ks}E^{[p,p+1]}_{sj}})_{ip}
= -f(\wt{E^{[p,p+1]}_{sj}})_{ip}E^{[pp]}_{ks}.
\end{equation}
However, we also have
\begin{eqnarray*}
0&=& f([\wt{E^{[p,p+1]}_{kj}}, \wt{E^{[p,p+1]}_{sj}}])_{i,p+1}
\\
&=& [f(\wt{E^{[p,p+1]}_{kj}}), \wt{E^{[p,p+1]}_{sj}}]_{i,p+1}+[\wt{E^{[p,p+1]}_{kj}}, f(\wt{E^{[p,p+1]}_{sj}})]_{i,p+1}
\\
&=&
f(\wt{E^{[p,p+1]}_{kj}})_{ip}  E^{[p,p+1]}_{sj} -f(\wt{E^{[p,p+1]}_{sj}})_{ip}E^{[p,p+1]}_{kj}
\\
&=&
 -f(\wt{E^{[p,p+1]}_{sj}})_{ip}E^{[pp]}_{ks}  E^{[p,p+1]}_{sj} -f(\wt{E^{[p,p+1]}_{sj}})_{ip}E^{[p,p+1]}_{kj}\qquad\text{(by \eqref{E_kj})}
\\
&=& -2 f(\wt{E^{[p,p+1]}_{sj}})_{ip}E^{[p,p+1]}_{kj}.
\end{eqnarray*}
Since ${\rm char}(\F)\ne 2$, the $k$-th column  of $f(\wt{E^{[p,p+1]}_{sj}})_{ip}$ must be zero.
Then  \eqref{E_kj} shows  that
$f(\wt{E^{[p,p+1]}_{kj}})_{ip}=-f(\wt{E^{[p,p+1]}_{sj}})_{ip}E^{[pp]}_{ks}=0.$

\item $i=p$:  \eqref{A_p(p+1)} shows that  for $A_{pp}\in \s\l_{pp}$ and $A_{p,p+1}\in \M_{p,p+1}$,
\begin{align*}
f(\wt{A_{pp}A_{p,p+1}})_{pp}&=[A_{pp}, f({\wt{A_{p,p+1}}})_{pp}]= A_{pp}f({\wt{A_{p,p+1}}})_{pp}-f({\wt{A_{p,p+1}}})_{pp}A_{pp}.
\end{align*}
In particular, for $r\in[m]-\{k\}$, we have $E_{kr}^{[pp]}\in{\s\l}_{pp}$ and
\begin{align}\label{B3}
f(\wt{E_{kj}^{[p,p+1]}})_{pp}=f(\wt{E_{kr}^{[pp]}}\wt{E_{rj}^{[p,p+1]}})_{pp}=E_{kr}^{[pp]}f(\wt{E_{rj}^{[p,p+1]}})_{pp}
-f(\wt{E_{rj}^{[p,p+1]}})_{pp}E_{kr}^{[pp]}.
\end{align}
Denote 
$$A=\mtx{a_{ij}}_{m\times m} :=f(\wt{E_{kj}^{[p,p+1]}})_{pp}.$$ 
\eqref{B3} implies that all nonzero entries of $A$ are located in the $k$-th row and  the $r$-th column. 
If $m\geq 3$,  we can replace $r$ by any $s\in [m]-\{k,r\}$ in  \eqref{B3} to show that
 all nonzero entries of $A$ are located in the $k$-th row. 
  In both $m=2$ and $m\ge 3$ cases, we have
\begin{equation}\label{E_kk^pp}
A=E_{kk}^{[pp]}A+a_{rr} E_{rr}^{[pp]}.
\end{equation}
Applying \eqref{B3} twice, we get
\begin{eqnarray*}
A
 &=& \left [E_{kr}^{[pp]}, f(\wt{E_{rj}^{[p,p+1]}})_{pp}\right] = \left[E_{kr}^{[pp]},\left[E_{rk}^{[pp]},f(\wt{E_{kj}^{[p,p+1]}})_{pp}\right]\right]
\\
 &=& E_{kk}^{[pp]}A-E_{kr}^{[pp]}A E_{rk}^{[pp]}
-E_{rk}^{[pp]}A E_{kr}^{[pp]}+A E_{rr}^{[pp]}
\\
 &=& (A-a_{rr} E_{rr}^{[pp]})-E_{kr}^{[pp]}(E_{kk}^{[pp]}A+a_{rr} E_{rr}^{[pp]})E_{rk}^{[pp]}
-E_{rk}^{[pp]}A E_{kr}^{[pp]}+A E_{rr}^{[pp]}
\quad\text{(by \eqref{E_kk^pp})}
\\
&=& 
A-a_{rr} (E_{rr}^{[pp]}+E_{kk}^{[pp]})-E_{rk}^{[pp]}AE_{kr}^{[pp]}+AE_{rr}^{[pp]}.
\end{eqnarray*}
Therefore,
\begin{align*}
a_{rr} (E_{rr}^{[pp]}+E_{kk}^{[pp]})+E_{rk}^{[pp]}A E_{kr}^{[pp]}=A E_{rr}^{[pp]}.
\end{align*}
Comparing the $(k,k)$ (resp. $(r,r)$, $(k,r)$) entry, we get $a_{rr} =0$ (resp. $a_{kk}=0$, $a_{kr}=0$).
Since $r\in [m]-\{k\}$ is arbitrary, we have $f(\wt{E_{kj}^{[p,p+1]}})_{pp}=0$.

\end{itemize}

We finish the proof for $\ell=1$. 

\item
Suppose \eqref{M_L^0-pq} is true for all $\ell<k$.  Now for any $(p,p+k)\in [I(\L)]$,
we have
$[\wt{\M_{p,p+1}},\wt{\M_{p+1,p+k}}]=\wt{\M_{p,p+k}}$ in $M_{\L}^0$,  and by induction hypothesis,
\begin{eqnarray*}
f(\wt{A_{p,p+1} A_{p+1,p+k}})
&=&f([\wt{A_{p,p+1}},\wt{A_{p+1,p+k}}])
= [f(\wt{A_{p,p+1}}),\wt{A_{p+1,p+k}}]+[\wt{A_{p,p+1}},f(\wt{A_{p+1,p+k}})]
\\
&\in&
 \wt{\M_{p,p+k}}+\sum_{i=1}^{p-1} \wt{\M_{i,p+k}}+\sum_{j=p+k+1}^{t} \wt{\M_{pj}}.
\end{eqnarray*}
Therefore,  \eqref{M_L^0-pq} is true for $\ell=k$.

\item
Overall,  \eqref{M_L^0-pq} is proved for all $(p,q)\in [I(\L)]$ with $p<q$.  
\qedhere

\end{enumerate}

\end{proof}

\begin{lemma}\label{existence and uniqueness}
Suppose $\rm {char}(\mathbb{F})\neq 2, 3$. Let $f\in \Der(M^{0}_{\L})$. Then
 for  any  $1\le p< q\le t$, there exists $X_{pq}\in \M_{pq}$ such that
\begin{align}
f(\wt{A_{ip}})_{iq} &=-A_{ip}X_{pq},\quad \text{for all}\ (i,p)\in [I(\L)] \text{ and } \wt{A_{ip}}\in \wt{\M_{ip}}\cap M_{\L}^0,
\label{M_L^0-A_ip}
\\
f(\wt{A_{qj}})_{pj} &=X_{pq}A_{qj},\quad \text{for all}\ (q,j)\in [I(\L)] \text{ and } \wt{A_{qj}}\in \wt{\M_{qj}}\cap M_{\L}^0.
\label{M_L^0-A_qj}
\end{align}
\end{lemma}

The proof is similar to part (3) of the proof of Theorem \ref{Main theorem} in Section \ref{sect:proof}.

\begin{proof}
Given $p<q$ in $[t]$, we consider the following four situations:
\begin{enumerate}
\item Suppose $(q,j)=(t,t)\in [I(\L)]$. For any $A_{tt}, B_{tt}\in {\s\l}_{tt}$,
\begin{align*}
f([\wt{A_{tt}}, \wt{B_{tt}}])_{pt}&=[f(\wt{A_{tt}}), \wt{B_{tt}}]_{pt}+[\wt{A_{tt}}, f(\wt{B_{tt}})]_{pt}=f(\wt{A_{tt}})_{pt}B_{tt}-f(\wt{B_{tt}})_{pt}A_{tt}.
\end{align*}
Applying Lemma \ref{thm:sl_n-case-2} to the map $\phi:{\s\l}_{tt}\to \M_{pt}$ defined by $\phi(C)=f(\wt{C})_{pt}$, we can find $X_{pt}\in \M_{pt}$ such that $f(\wt{A_{tt}})_{pt}=X_{pt}A_{tt}$ for $A_{tt}\in {\s\l}_{tt}$.

\item
Similarly, when $(i,p)=(1,1)$, there exists $Y_{1q}\in \M_{1q}$ such that $f(\wt{A_{11}})_{1q}=-A_{11}Y_{1q}$ for all $A_{11}\in {\s\l}_{11}$.

\item
Suppose $(q,j)\in [I(\L)]$, $(q,j)\neq (t,t)$.
Then $q<t$.  Given any $j< j'$ in $[t]$, we have $(j,j'), (q,j'), (p,j),(p,j')\in [I(\L)]$, and $\wt{\M_{qj'}}=\wt{\M_{qj}}\wt{\M_{jj'}}=[\wt{\M_{qj}},\wt{\M_{jj'}}]$.

\begin{itemize}
\item If $q=j$,  then for $A_{qj}\in{\s\l}_{qq}$ and $A_{jj'}\in\M_{jj'}$,
$$
f(\wt{A_{qj} A_{jj'}})_{pj'}
=f([\wt{A_{qj}},\wt{A_{jj'}}])_{pj'}=[f(\wt{A_{qj}}),\wt{A_{jj'}}]_{pj'}+[\wt{A_{qj}},f(\wt{A_{jj'}})]_{pj'}
=f(\wt{A_{qj}})_{pj}A_{jj'}.
$$
Applying Lemma \ref{thm:block-tran-4} to the map $\phi:\M_{qj'}\to \M_{pj'}$ defined by $\phi(C)=f(\wt{C})_{pj'}$, and $\varphi:{\s\l}_{qq}\to \M_{pq}$ defined by $\varphi(D)=f(\wt{D})_{pj}$, there exists $X_{pq}\in \M_{pq}$ such that $f(\wt{A_{qj}})_{pj}=X_{pq}A_{qj}$ for $A_{qj}\in {\s\l}_{qq}$,
and $f(\wt{A_{qj'}})_{pj'}=X_{pq}A_{qj'}$ for  any $j'>j$ in $[t]$ and any $A_{qj'}\in {\M}_{qj'}$.

\item If $q<j$, then for $A_{qj}\in{\M}_{qj}$ and $A_{jj'}\in\M_{jj'}$, we still have
$$
f(\wt{A_{qj} A_{jj'}})_{pj'}
=f([\wt{A_{qj}},\wt{A_{jj'}}])_{pj'}=[f(\wt{A_{qj}}),\wt{A_{jj'}}]_{pj'}+[\wt{A_{qj}},f(\wt{A_{jj'}})]_{pj'}
=f(\wt{A_{qj}})_{pj}A_{jj'}.
$$
Applying Lemma \ref{thm:block-tran-2},  there exists a (unique)  $X_{pq}\in \M_{pq}$ such that $f(\wt{A_{qj}})_{pj}=X_{pq}A_{qj}$ for all $j>q$ in $[t]$.
\end{itemize}

\item Suppose $(i,p)\in [I(\L)]$ and $(i,p)\neq (1,1)$. Similar to the proceeding argument, there exists $-Y_{pq}\in \M_{pq}$ such that $f(\wt{A_{ip}})_{iq}=-A_{ip}Y_{pq}$ for $(i,p)\in [I(\L)]$ and $A_{ip}\in{\M}_{ip}$.

\item For any $(i,p), (q,j)\in [I(\L)]$, we have $[\wt{A_{ip}}, \wt{A_{qj}}]=0$. So
\begin{eqnarray*}
0&=&f([\wt{A_{ip}},\wt{A_{qj}}])_{ij}
=[f(\wt{A_{ip}}),\wt{A_{qj}}]_{ij}+[\wt{A_{ip}},f(\wt{A_{qj}})]_{ij}
\\
&=& f(\wt{A_{ip}})_{iq}A_{qj}+A_{ip}f(\wt{A_{qj}})_{pj}
=-A_{ip}Y_{pq}A_{qj}+A_{ip}X_{pq}A_{qj}.
\end{eqnarray*}
Therefore, $X_{pq}=Y_{pq}$.
\qedhere
\end{enumerate}

\end{proof}


Now we are ready to prove  Theorem \ref{thm:strongly dominant}.

\begin{proof}[Proof of Theorem \ref{thm:strongly dominant}]
We have the Lie subalgebra decomposition
$$
M_{\L}={\rm span}\{ E_{11}^{[kk]}\mid (k,k)\in [I(\L)]\}\ltimes M_{\L}^0.
$$
Given $f\in\Der(M_{\L}^0)$, we define $f^+ (A):=f(A)$ for $A\in M_{\L}^0$.  The next step is to define $f^+(\wt{E_{11}^{[kk]}})$  for each $(k,k)\in [I(\L)]$ appropriately
so that  $f^+\in\Der(M_{\L})$.
We will let
$$
f^+(\wt{E_{11}^{[kk]}})\ \in\ \wt{\s\l_{kk}}+\sum_{i=1}^{k-1} \wt{\M_{ik}}+\sum_{j=k+1}^{t} \wt{\M_{kj}}
$$
and define the nonzero blocks of $f^+(\wt{E_{11}^{[kk]}})$ as follow.   

\begin{enumerate}

\item The $(k,k)$ block: it is easy to see that $f(\wt{\cdot})_{kk}: \s\l_{kk}\to \s\l_{kk}$, $A_{kk}\mapsto f(\wt{A_{kk}})_{kk}$, is a derivation of $\s\l_{kk}$.
Since ${\rm char}(\F)\ne 2$, there exists $X_{kk}\in\s\l_{kk}$ such that $ f(\wt{A_{kk}})_{kk}=[X_{kk}, A_{kk}]$ for $A_{kk}\in\s\l_{kk}$.
Define
\begin{equation}\label{def-kk}
f^+(\wt{E_{11}^{[kk]}})_{kk}:=[X_{kk}, E_{11}^{[kk]}].
\end{equation}

\item The $(i,k)$ block, $i<k$:
by Lemma \ref{existence and uniqueness},  there exists $X_{ik}\in \M_{ik}$ such that $f(\wt{A_{kj}})_{ij}=X_{ik}A_{kj}$ for any $(k,j)\in [I(\L)]$. Define
        \begin{align}\label{def-ik}
        f^{+}(\wt{E^{[kk]}_{11}})_{ik}:=X_{ik}E^{[kk]}_{11}\quad \text {for all }  i\in [k-1].
        \end{align}

\item The $(k,j)$ block, $k<j$: by Lemma \ref{existence and uniqueness}, there exists $X_{kj}\in \M_{kj}$ such that for all $(i,k)\in [I(\L)]$ we have $f(\wt{A_{ik}})_{ij}=-A_{ik}X_{kj}$. Define
        \begin{align}\label{def-kj}
        f^{+}(\wt{E^{[kk]}_{11}})_{kj}:=-E^{[kk]}_{11}X_{kj}\quad \text {for all }  k<j\leq t.
        \end{align}

\end{enumerate}

The above process uniquely defines a linear map $f^{+}\in \rm {End(M_{\L})}$ such that $f^{+}|_{M^{0}_{\L}}=f$.
Next we verify that $f^{+}\in \rm {Der(M_{\L})}$. It suffices to prove that for every $(i,j)\in [I(\L)]$,
\begin{equation}\label{f-E11-1}
f^{+}([\wt {E^{[kk]}_{11}}, \wt{A_{ij}}])=[f^{+}(\wt {E^{[kk]}_{11}}), \wt{A_{ij}}]+[\wt {E^{[kk]}_{11}}, f^{+}(\wt{A_{ij}})]\quad \text{for all}\ \wt{A_{ij}}\in \wt{\M_{ij}}\cap M^{0}_{\L}.
\end{equation}
Denote
\begin{equation}\label{X_k}
X_k:=\wt{X_{kk}}+\sum_{i=1}^{k-1} \wt{X_{ik}}+\sum_{j=k+1}^t \wt{X_{kj}}.
\end{equation}
Then \eqref{def-kk}, \eqref{def-ik},  and \eqref{def-kj} imply that $f^+(\wt{E^{[kk]}_{11}})=[X_k,\wt{E^{[kk]}_{11}}]$. So
\eqref{f-E11-1} is equivalent to
\begin{equation}\label{f-E11-2}
f([\wt {E^{[kk]}_{11}}, \wt{A_{ij}}])=[[X_k,\wt{E^{[kk]}_{11}}], \wt{A_{ij}}]+[\wt {E^{[kk]}_{11}}, f(\wt{A_{ij}})]\quad \text{for all}\ \wt{A_{ij}}\in \wt{\M_{ij}}\cap M^{0}_{\L}.
\end{equation}
We will prove \eqref{f-E11-2} for each block $(i,j)\in [I(\L)]$:

\begin{enumerate}

\item $(k,k)\in [I(\L)]$:   the matrices $X_{kk}$, $X_{ik}$ ($i<k$), and $X_{kj}$ ($k<j$)   satisfy that
$$
f(\wt{A_{kk}})=[X_k,\wt{A_{kk}}]\quad \text{for all } A_{kk}\in\s\l_{kk},
$$
where $X_k$ is given by \eqref{X_k}.
Therefore, \eqref{f-E11-2} is true for $(i,j)=(k,k)\in [I(\L)]$.

\item $(k,j)$, $k<j\leq t$:
when $(i,j)=(k,j)$, we have
$$
[\wt {E^{[kk]}_{11}}, \wt{A_{kj}}]=\wt {E^{[kk]}_{11}} \wt{A_{kj}}=\wt {E^{[kk]}_{12}} \wt {E^{[kk]}_{21}} \wt{A_{kj}}=
[\wt {E^{[kk]}_{12}},[\wt {E^{[kk]}_{21}}, \wt{A_{kj}}]].
$$
So \eqref{f-E11-2} is equivalent to the following equalities:
\begin{eqnarray*}
&&f([\wt {E^{[kk]}_{12}},[\wt {E^{[kk]}_{21}}, \wt{A_{kj}}]])=[[X_k,\wt{E^{[kk]}_{11}}], \wt{A_{kj}}]+[\wt {E^{[kk]}_{11}}, f(\wt{A_{kj}})]
\\
  &\Longleftrightarrow & f(\wt {E^{[kk]}_{12}})\wt {E^{[kk]}_{21}}\wt{A_{kj}}+\wt {E^{[kk]}_{12}}f(\wt {E^{[kk]}_{21}})\wt{A_{kj}}+\wt {E^{[kk]}_{12}}\wt {E^{[kk]}_{21}}f(\wt{A_{kj}})=[X_k,\wt{E^{[kk]}_{11}}]\wt{A_{kj}}+E^{[kk]}_{11}f(\wt{A_{kj}})
  \\
 &\Longleftrightarrow &  f(\wt {E^{[kk]}_{12}})\wt {E^{[kk]}_{21}}\wt{A_{kj}}+\wt {E^{[kk]}_{12}}f(\wt {E^{[kk]}_{21}})\wt{A_{kj}}=[X_k,\wt{E^{[kk]}_{11}}]\wt{A_{kj}}\quad (\text{for all $A_{kj}\in \M_{kj}$})
 \\
 &\Longleftrightarrow & f(\wt {E^{[kk]}_{12}})\wt {E^{[kk]}_{21}}+\wt {E^{[kk]}_{12}}f(\wt {E^{[kk]}_{21}})=[X_k,\wt{E^{[kk]}_{11}}]
 \\
 &\Longleftrightarrow &[X_k,\wt{E^{[kk]}_{12}}]\wt {E^{[kk]}_{21}}+\wt {E^{[kk]}_{12}}[X_k,\wt{E^{[kk]}_{21}}]=[X_k,\wt{E^{[kk]}_{11}}].
\end{eqnarray*}
The last equality is obviously true.

\item $(i,k)$, $1\le i<k$: similarly, we can prove  \eqref{f-E11-2} for the case $(i,j)=(i,k)$.

\item $(i,j)\in [I(\L)]$, $i\ne k$, $j\ne k$:
the left side of \eqref{f-E11-2} is zero.  We investigate  the right side of  \eqref{f-E11-2} in three cases:
\begin{enumerate}
 \item $i\le j<k$: the only possibly nonzero block in  the right side of  \eqref{f-E11-2} is the $(i,k)$ block, which is
\begin{align*}
[[X_k,\wt{E^{[kk]}_{11}}], \wt{A_{ij}}]_{ik}+[\wt {E^{[kk]}_{11}}, f(\wt{A_{ij}})]_{ik}
&= -A_{ij}[X_k,\wt{E^{[kk]}_{11}}]_{jk}-f(\wt{A_{ij}})_{ik}E^{[kk]}_{11}
\\
&= -A_{ij}[X_k,\wt{E^{[kk]}_{11}}]_{jk}+A_{ij}X_{jk}E^{[kk]}_{11}\ \ \text{(by Lemma \ref{existence and uniqueness})}
\\
&= -A_{ij}X_{jk}E^{[kk]}_{11}+A_{ij}X_{jk}E^{[kk]}_{11}\ \ \text{(by \eqref{X_k})}
\\
&= 0.
\end{align*}
So  \eqref{f-E11-2} is done for this case.

\item $k<i\leq j$: similarly, we can prove \eqref{f-E11-2} for this case.

\item  $i<k<j$:  the right side of  \eqref{f-E11-2} is
$$
[[X_k,\wt{E^{[kk]}_{11}}], \wt{A_{ij}}]+[\wt {E^{[kk]}_{11}}, f(\wt{A_{ij}})]=0+0=0.
$$
So \eqref{f-E11-2} holds.

\end{enumerate}

\end{enumerate}

Overall, we have proved \eqref{f-E11-2}.  Therefore, $f^+\in\Der(M_{\L})$ and  $f^+|_{M_{\L}^0}=f$.
By Theorem \ref{Main theorem}, there is $X\in M_{\L_B}$ such that
$f(B)=[X, B]$ for all $B\in M_{\L}^0$.
\end{proof}


\end{document}